\documentclass[10pt]{amsart}
 
\usepackage{amsthm,amsmath,amssymb}
\usepackage{graphicx}
\usepackage[comma, sort&compress]{natbib}
\usepackage{paralist}
\usepackage[hyperindex,breaklinks]{hyperref}
\usepackage{kbordermatrix} % for typing border matrix
\usepackage{makecell}
\usepackage{tikz}
\usetikzlibrary{shapes,arrows,positioning}
\graphicspath{}

\newtheorem{theorem}{Theorem}

\newtheorem{lemma}[theorem]{Lemma} 
\newtheorem{proposition}[theorem]{Proposition}
\newtheorem{corollary}[theorem]{Corollary}

\theoremstyle{definition} 
\newtheorem{definition}{Definition}
\theoremstyle{remark} 

\newtheorem*{remark}{Remark}
\newtheorem*{claim*}{Claim}

\newcommand{\figref}[1]{Figure~\ref{fig:#1}}

\newcommand{\secref}[1]{Section~\ref{sec:#1}}

\newcommand{\defref}[1]{Definition~\ref{def:#1}}

\newcommand{\lemref}[1]{Lemma~\ref{lem:#1}}
\newcommand{\lemsref}[1]{Lemmas~\ref{lem:#1}}
\newcommand{\lemssref}[1]{\ref{lem:#1}}
\newcommand{\propref}[1]{Proposition~\ref{prop:#1}}

\newcommand{\propssref}[1]{\ref{prop:#1}}
\newcommand{\thmref}[1]{Theorem~\ref{thm:#1}}

\newcommand{\thmsref}[1]{Theorems~\ref{thm:#1}}
\newcommand{\thmssref}[1]{\ref{thm:#1}}
\newcommand{\corref}[1]{Corollary~\ref{cor:#1}}
\newcommand{\tabref}[1]{Table~\ref{tab:#1}}

\numberwithin{equation}{section}
\numberwithin{theorem}{section}
\numberwithin{example}{section}
\numberwithin{definition}{section}
\numberwithin{figure}{section}

\DeclareMathOperator{\diag}{diag}

\makeatletter
\def\@strippedMR{}
\def\@scanforMR#1#2#3\endscan{%
  \ifx#1M\ifx#2R\def\@strippedMR{#3}%
  \else\def\@strippedMR{#1#2#3}%
  \fi\fi}
\renewcommand\MR[1]{\relax
  \ifhmode\unskip\spacefactor3000 \space\fi
  \@scanforMR#1\endscan
  MR\MRhref{\@strippedMR}{\@strippedMR}}
\makeatother

\title[]{Identifiability of directed Gaussian graphical models with one latent source}

\author[D.~Leung]{Dennis Leung} 
\address{Department of Statistics, University of Washington, Seattle,
  WA, U.S.A.}
\email{dmhleung@uw.edu}

\author[M.~Drton]{Mathias Drton} 
\address{Department of Statistics, University of Washington, Seattle,
  WA, U.S.A.}
\email{md5@uw.edu}

\author[H.~Hara]{Hisayuki Hara} 
\address{Faculty of Economics, University of Niigata, Niigata, Japan}
\email{hara@econ.niigata-u.ac.jp}

\date{Draft from 19 January 2015}

\begin{document}
\begin{abstract}
  We study parameter identifiability of directed Gaussian graphical
  models with one latent variable.  In the scenario we consider, the
  latent variable is a confounder that forms a source node of the
  graph and is a parent to all other nodes, which correspond to the
  observed variables.  We give a graphical condition that is
  sufficient for the Jacobian matrix of the parametrization map to be
  full rank, which entails that the parametrization is generically
  finite-to-one, a fact that is sometimes also referred to as local
  identifiability.  We also derive a graphical condition that is
  necessary for such identifiability.  Finally, we give a condition
  under which generic parameter identifiability can be determined from
  identifiability of a model associated with a subgraph.  The power of
  these criteria is assessed via an exhaustive algebraic computational
  study on models with 4, 5, and 6 observable variables.
\end{abstract}



\maketitle

\section{Introduction} \label{sec:intro}

In this paper we study parameter identifiability in directed Gaussian
graphical models with a latent variable.  Our work falls in a line of
work where the graphical representation of causally interpretable
latent variable models is used to give tractable criteria to decide
whether parameters can be uniquely recovered from the joint
distribution of the observed variables \citep{PearlCausality}.  Some
examples of prior work in this context are \cite{chen2014},
\cite{drton2011global}, \cite{foygel2012}, \cite{Grzebyk2004},
\cite{kuroki:2004}, \cite{kuroki:2014}, \cite{wermuth2005},
\cite{tian}, and \cite{tian:2009}.

The setup we consider has a single latent variable appear as a source
node in the directed graph defining the Gaussian model.  The resulting
models can be described as follows.  Let $X_1, \dots, X_m$ be
observable variables, and let $L$ be a hidden variable, and suppose
the variables are related by linear equations as
\begin{equation*}
X_v = \sum_{w \not= v} \lambda_{wv} X_w  + \delta_v L + \epsilon_v , \qquad v = 1, \dots, m,
\end{equation*}
where $\lambda_{wv}$, $\delta_v$ are real coefficients quantifying linear
relationships, and the $\epsilon_v$ are independent mean zero Gaussian
noise terms with variances $\omega_v>0$.  The latent variable $L$ is
assumed to be standard normal and independent of the noise terms
$\epsilon_v$.  Letting $X = (X_1, \dots, X_m)^T$,
$\epsilon = (\epsilon_1, \dots, \epsilon_m)^T$ and
$\delta = (\delta_1, \dots, \delta_m)^T$, we may present the model in
the vectorized form
\begin{equation}\label{Matsys}
X  = \Lambda^T X + \delta L + \epsilon,
\end{equation}
where $\Lambda$ is the matrix $(\lambda_{wv})$ with $\lambda_{vv}=0$
for all $v=1,\dots,m$.  We are then interested in specific models, in
which for certain pairs of nodes $w\not=v$ the coefficient
$\lambda_{wv}$ is constrained to zero.  In particular, we are
interested in recursive models, that is, models in which the matrix
$\Lambda$ can be brought into strictly upper triangular form by
permuting the indices of the variables (and thus the rows and columns
of $\Lambda$).  This implies that $I_m - \Lambda$ is invertible,
where $I_m$ is the $m \times m$ identity matrix. It follows that the
observable variate vector $X$ has a $m$-variate normal distribution
$N_m(0, \Sigma)$ with covariance matrix
\begin{equation} \label{sigmacov}
\Sigma = (I_m - \Lambda^T)^{-1} (\Omega + \delta \delta^T)(I_m - \Lambda)^{-1},
\end{equation}
where $\Omega$ is the diagonal matrix with $\Omega_{vv} = \omega_v$.
For additional background on graphical models we refer the reader to
\cite{lau1996} and \cite{PearlCausality}.  We note that the models we
consider also belong to the class of linear structural equation models
\citep{bollen:1989}.

A Gaussian latent variable model postulating recursive zero structure
in the matrix $\Lambda$ from \eqref{Matsys} can be thought of as
associated with a graph $G = (V, E)$ whose vertex set
$V = \{1, \dots, m\}$ is the index set for the observable variables
$X_1, \dots, X_m$.  For two distinct nodes $w,v\in V$, the edge set
$E$ includes the directed edge $(w, v)$, denoted as $w \rightarrow v$
if and only if the model includes $\lambda_{wv}$ as a free parameter.
When the model is recursive, the directed graph $G$ is acyclic and
following common terminology we refer to $G$ as a DAG (for directed
acyclic graph).  In this paper, we will then always assume that the
nodes are labeled in topological order, that is, we have
$V=\{1,\dots,m\}$ and $w \rightarrow v \in E$ only if $w < v$.  
% We
% remark that a DAG can have more than one topological order but this
% has no bearing on our later results.

To emphasize the presence of the latent variable $L$, one could
equivalently represent the model by an extended DAG
$\overline{G} = (\overline{V}, \overline{E})$ on $m+1$ nodes
enumerated as $\overline{V} := \{0, 1, \dots, m\}$, where the node $0$
corresponds to the latent variable $L$, and if $G = (V, E)$ is the
graph on $m$ nodes representing the model in the preceding paragraph,
then
$\overline{E} = E \cup \{0 \rightarrow v : v \in \{1, \dots, m\}\}$.
The edges $0 \rightarrow v$ correspond to the coefficients $\delta_v$.

For the DAG $G=(V,E)$, let
\[
\mathbb{R}_E := \{\Lambda = (\lambda_{wv}) \in \mathbb{R}^{m \times
  m} : w \rightarrow v \not \in E \Rightarrow  \lambda_{wv}= 0 \}
\]
be the linear space of coefficient matrices, and let $\diag_m^+$ be the set
of all $m \times m$ diagonal matrices with a positive diagonal.

%% Formal definition
\begin{definition}\label{def:G*mod}
  The \emph{Gaussian one latent source model} associated with a given DAG
  $G = (V, E)$, denoted as $\mathcal{N}_{*}(G)$, is the family of all
  $m$-variate normal distributions $N_m(0, \Sigma)$ with a covariance
  matrix of the form
\[\Sigma \;=\; (I_m - \Lambda^T)^{-1} (\Omega + \delta \delta^T)(I_m - \Lambda)^{-1}, \]
for $\Lambda \in \mathbb{R}_E$, $\Omega \in \diag_m^+$ and
$\delta \in \mathbb{R}^m$.
\end{definition}

 The model $\mathcal{N}_*(G)$ has the parametrization map
\begin{alignat} {2} \label{phi}
\phi_{G}  &\;:\;& \; 
% \Theta &\longrightarrow \text{PD}_m ,\\
% \notag
% && \quad 
(\Lambda, \Omega, \delta) &\longmapsto ( I_m - \Lambda^T)^{-1} (\Omega + \delta \delta^T)(I_m - \Lambda)^{-1}
\end{alignat}
defined on the set
$\Theta := \mathbb{R}_E \times \diag_m^+ \times \mathbb{R}^m$, which
we may also view as an open subset of $\mathbb{R}^ {2m + |E|}$, where
$|E|$ is the cardinality of the directed edge set $E$.  Clearly, the
image of $\phi_G$ is in $\text{PD}_m$, the cone of positive definite
$m \times m$ matrices.  Note that since $G$ is acyclic, we have
$(I_m - \Lambda)^{-1} = I_m + \Lambda + \Lambda^2 + \dots + \Lambda^{m
  -1}$ and thus the covariance parametrization $\phi_{G}$ is a polynomial map.

In this paper we will derive graphical conditions on $G$ that are
sufficient/necessary for identifiability of the model
$\mathcal{N}_* (G)$. 
% Specifically, we ask, ``Under what conditions on
% $G$ is the parametrization map in \eqref{phi} injective in a certain
% sense?.''  
We begin by clarifying what precisely we mean by identifiability.
The most stringent notion, namely that of global identifiability,
requires $\phi_G$ to be injective on all of $\Theta$.  While this
notion is important \citep{drton2011global}, it  is too
stringent for the setting we consider here.  Indeed, for any triple
$(\Lambda, \Omega, \delta) \in \Theta$,
$\phi_{G}( \Lambda, \Omega, \delta) = \phi_{G}(\Lambda, \Omega,
-\delta)$,
which implies that the \emph{fiber}
\[
\left\{(\Lambda', \Omega', \delta') \in \Theta : \phi_{G} (\Lambda, \Omega, \delta)  = \phi_{G} \left(\Lambda', \Omega', \delta'\right) \right\}
\] 
always has cardinality $\geq 2$. We may account for this symmetry by
requiring $\phi_G$ to be 2-to-1 on all of $\Theta$ but this is not
enough as there are always some fibers that are infinite.  For
instance, it is easy to show that the fiber in the above display is
infinite when $\delta=0$.  As such, it is natural to consider notions
of generic identifiability.  Specifically, our contributions will
pertain to the notion of generic finite identifiability, as defined
below, that only requires finite identification of parameters away
from a fixed null set in $\Theta$; here a null set is a set of
Lebesgue measure zero.  This notion is also referred to as local
identifiability in other related work such as \cite{FactBerkeley}.

Null sets appearing in our work are algebraic sets, where an
algebraic set $A \subset \mathbb{R}^n$ is the set of
common zeros of a collection of multivariate polynomials, i.e.,
\[A = \{a \in \mathbb{R}^n : f_i (a) = 0,\; i = 1, \dots, k\},\]
for $f_i \in \mathbb{R}[x_1, \dots, x_n]$, where
$\mathbb{R}[x_1, \dots, x_n]$ is the ring of polynomials in $n$
variables with coefficients in $\mathbb{R}$.  Note that $A$ is a
closed set in the usual Euclidean topology.  If all polynomials $f_i$
are the zero polynomial then $A=\mathbb{R}^n$.  Otherwise, $A$ is a
proper subset, $A\subsetneq \mathbb{R}^n$, and its dimension is then
less than $n$.  In particular, a proper algebraic subset of
$\mathbb{R}^n$ has measure zero.

\begin{definition}\label{def:gfinite}
  Let $S$ be an open subset of $\mathbb{R}^n$, and let $f$ be a map
  defined on $S$. Then $f$ is said to be \emph{generically
    finite-to-one} if there exists a proper algebraic set
  $\tilde{S} \subset \mathbb{R}^n$ such that the \emph{fiber} of $s$,
  i.e.~the set $\{ s' \in S: f(s') = f(s) \}$, is finite for all
  $s \in S \setminus \tilde{S}$. Otherwise, $f$ is said to be
  generically infinite-to-one.
\end{definition}

\begin{definition} \label{def:gfiden} The model $\mathcal{N}_*(G)$ of
  a given DAG $G = (V, E)$ is said to be \emph{generically finitely
    identifiable} if its parametrization $\phi_G$ defined on
  $\Theta$ is generically finite-to-one. We also say the DAG $G$
  is generically finitely identifiable for short.
\end{definition}

Hereafter for any map $f$ defined on an open domain $S \subset \mathbb{R}^n$, we will use 
\begin{equation}
  \label{eq:fiber:notation}
\mathcal{F}_f(s) := \{s' \in S: f(s') = f(s)\}
\end{equation}
to denote the fiber of a point $s \in \mathcal{S}$. If $T$ is a
subset of $S$, we will use $f|_T$ to denote the restriction of $f$ to
$T$, in which case for any $t \in T$, we have the fiber 
$$
\mathcal{F}_{f|_T}(t) = \{t' \in T: f(t') = f(t)\}.
$$
 % When $\tilde{S}$ is a proper algebraic subset of $\mathbb{R}^n$ not
 % necessarily contained in $S$, we may abuse the inclusion notation
 % ``$\subset$'' slightly by saying ``$\tilde{S} \subset S$''. 
The term ``generic point" will refer to any point in the domain $S$
that lies outside a fixed proper algebraic subset
$\tilde{S}$, and a property is said to hold generically if
it holds everywhere on $S\setminus\tilde{S}$.  The
following well-known lemma is a main tool in this paper, and its proof
will be included in Appendix~\ref{sec:proofs} for completeness. It gives as an
immediate corollary a trivial necessary condition for generic finite
identifiability.

\begin{lemma}\label{lem:localequalfinite}

Suppose $f: S \rightarrow \mathbb{R}^d$ is a polynomial map defined on
an open set $S \subset \mathbb{R}^n$. The following statements are equivalent:
\begin{enumerate}\item $f$ is generically finite-to-one.
\item There exists a proper algebraic subset
  $\tilde{S} \subset \mathbb{R}^n$ such that the fibers of the
  restricted map $f|_{S \setminus \tilde{S}}$ are all finite,
  i.e.~$|\mathcal{F}_{ f|_{S \setminus \tilde{S}}}(s)| < \infty$ for
  all $s \in S \setminus \tilde{S}$. 
%\[ \mathcal{F}_{ f|_{S \setminus \tilde{S}}}(s) 
% :=  \{s' \in S \setminus \tilde{S}: f(s') = f(s)\}
%\]

\item The Jacobian matrix of $f$  is generically of full column rank.
\end{enumerate}
\end{lemma}
%
%\begin{corollary}
%Any one of the maps $\phi_G$, $\tilde{\phi}_G$, $\varphi_G$ and $\tilde{\varphi}_G$ is generically finite-to-one if and only if it is generically locally injective. Hence the one factor model associated with $G$ is generically finitely identifiable if and only if it is generically locally identifiable.
%\end{corollary}
%%
%
%
% Note the subtle difference between $(i)$ and $(ii)$ in relation to \defref{gfinite}. Although $(ii)$ appears to be weaker than $(i)$, it is useful in proving one of our results later.  As an immediate corollary, \lemref{localequalfinite} gives the following necessary condition for generic finite identifiability of a $G^*$ model.

\begin{corollary}\label{cor:basicthm}
Given a DAG $G = (V, E)$, a necessary condition for generic finite
identifiability of its associated model $\mathcal{N}_*(G)$ is that ${m+1 \choose 2} - 2m \geq |E|$.
% or equivalently, if we define $N$ to be the set of ``non-edges", i.e. $N  :=  \{(v, w) : 1\leq v<w \leq m \text{ and } v \rightarrow w \not \in E\}$, $|N| \geq m$.
\end{corollary}

\begin{proof} [Proof of \corref{basicthm}]
The Jacobian matrix of $\phi_G$ is of size ${m +1 \choose 2} \times (|E| + 2m)$, and it is necessary that ${m +1  \choose 2} \geq |E| + 2m$ for it to have full column rank.
\end{proof}

Property $(ii)$ is seemingly weaker than $(i)$ in
\lemref{localequalfinite}. It is useful in proving our
results in \secref{hara}. In light of \corref{basicthm}, for the rest
of this paper we will restrict our attention to DAGs $G = (V, E)$ with
${m + 1 \choose 2} - 2m \geq |E|$, in which case $m$ must be at least
$3$.

One of our contributions is a sufficient graphical condition stated in
\thmref{our_suff} below. For $v \not = w \in V$, we will use
$v \text{ --- } w$ or $w \text{ --- } v$ to denote the edge
$(v, w) = (w, v)$ of an undirected graph on $V$. With slight abuse of
notation, we may also use $v\text{ --- }w$ or $ w\text{ --- }v$ to
denote an edge $v \rightarrow w \in E$ when the directionality of
edges in a DAG $G = (V, E)$ is to be ignored. For any
directed/undirected graph $G = (V, E)$, the complement of $G$, denoted
as $G^c = (V, E^c)$, is the undirected graph on $V$ with the
edge set
$E^c = \{v \text{ --- } w:  (v,w)\not\in E  \text{ and  } (w,v) \not \in E\}$.

\begin{theorem} [Sufficient condition for generic finite identifiability]\label{thm:our_suff} 
 The model $\mathcal{N}_*(G)$ given by a DAG $G = (V, E)$ 
%% in \defref{G*mod} 
is generically finitely identifiable if every connected component of $G^c$ contains an odd cycle.
\end{theorem}

\figref{graph675} shows a DAG $G$ that satisfies the sufficient condition in
\thmref{our_suff}; its undirected complement $G^c$ is shown on the
right of the figure.  We will revisit this example in
Section~\ref{sec:example}, where we report on algebaric computations
that show that for this graph $G$ the fibers of $\phi_G$ are
generically of size 2 or 4.

\begin{figure}[t] 
\centering

\ \

\ \ 
 \begin{minipage}[t]{0.45\textwidth}
\centering
 \begin{tikzpicture}[->,>=triangle 45, shorten >=0pt,
        auto,thick,
        main node/.style={circle,inner sep=2pt,fill=gray!20,draw,font=\sffamily}]

        \node[main node,rounded corners] (1) {1};
        \node[main node,rounded corners] (2)
        [below =.5cm and 0cm of 1] {2};
        \node[main node,rounded corners] (3)
        [right=0cm and 1.5 cm of 1] {3};
        \node[main node,rounded corners] (4)
        [below = .5 cm  of 3] {4};
        \node[main node,rounded corners] (5)
        [right =1.5 cm of 4] {5};
 
        \path[color=black!20!blue,->,every
        node/.style={font=\sffamily\small}]
        (1) edge node {} (3)
        (1) edge node {} (2)
        (2)  edge node {} (4)
       (3) edge node {} (4) 
        (4) edge node {} (5);
 \end{tikzpicture}

\end{minipage}
\begin{minipage}[t]{0.45\textwidth}
\centering

 \begin{tikzpicture}[->,>=triangle 45, shorten >=0pt,
             auto,thick,
        main node/.style={circle,inner sep=2pt,fill=gray!20,draw,font=\sffamily}]

        \node[main node,rounded corners] (1) {1};
        \node[main node,rounded corners] (2)
        [below =.5cm and 0cm of 1] {2};
        \node[main node,rounded corners] (4)
        [right=0cm and 1.5 cm of 1] {4};
        \node[main node,rounded corners] (3)
        [below = .5 cm  of 4] {3};
        \node[main node,rounded corners] (5)
        [right =1.5 cm of 4] {5};
 
        \path[color=black!20!red,-,every
        node/.style={font=\sffamily\small}]
        (1) edge node {} (2)
        (1) edge node {} (4)
        (1)  edge [bend left= 20]node {} (5)
       (2) edge node {} (3) 
        (3) edge node {} (5)
      (2) edge node {} (5);
 \end{tikzpicture}
\end{minipage}
\caption{A DAG $G$ that satisfies the sufficient condition in
  \thmref{our_suff}; its undirected complement $G^c$ is shown on the
  right.  
  % The complex parametrization $\phi_G^{\mathbb{C}}$ is
  % generically 4-to-1.
}
 \label{fig:graph675}
\end{figure}
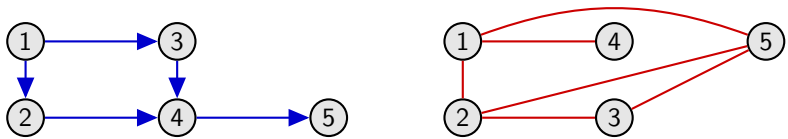

Our approach to proving \thmref{our_suff} also yields a necessary
condition for generic finite identifiability. This condition can be
stated in terms of two undirected graphs on the node set $V$, denoted
$G_{|L, cov}= (V, E_{|L, cov})$ and $G_{con}= (V, E_{con})$, where
$E_{|L, cov}$ captures the dependency of variable pairs after
conditioning on the latent variable $L$, and $E_{con}$ captures the
dependency of variable pairs after conditioning on all other
variables. From \eqref{Matsys} it can be seen that
$\Sigma_{|L} := ( I_m - \Lambda^T)^{-1} \Omega (I_m - \Lambda)^{-1}$
is the covariance matrix of $X$ conditioning on $L$, hence
$v \text{ --- } w\in E_{|L, cov}$ if and only if
$(\Sigma_{|L})_{vw} \not \equiv 0$, and analogously
$v \text{ --- } w\in E_{con}$ if and only if
$(\Sigma_{|L}^{-1})_{vw} \not \equiv 0$. It is well known that these
two undirected graphs can be obtained by using the d-separation
criterion applied to the extended DAG $\overline{G}$; see
\citet[p.~73]{LAS} for example.

\begin{theorem} [Necessary condition for generic finite identifiability]\label{thm:our_nec}
Given a DAG $G = (V, E)$, for the model $\mathcal{N}_*(G)$ to be generically finitely identifiable, it is necessary that the following two conditions both hold:
\begin{enumerate}
\item $|E_{con}| -  |E| \geq d_{con}$, where $d_{con}$ is the number of connected components in the graph $(G_{con})^c$ that do not contain any odd cycle;
\item $|E_{|L, cov}| - |E| \geq d_{cov}$, where $d_{cov}$ is the
  number of connected components in the graph $(G_{|L, cov})^c$ that do
  not contain any odd cycle.
\end{enumerate}
\end{theorem}

\figref{graph7539} gives an example of a DAG that fails to satisfy our
necessary condition, specifically, condition $(ii)$.

\begin{figure}[t] 
\centering
\ \

\ \ 
\begin{minipage}[t]{0.3\hsize}
\centering
 \begin{tikzpicture}[baseline=(1),->,>=triangle 45, shorten >=0pt,
        auto,thick,
        main node/.style={circle,inner sep=2pt,fill=gray!20,draw,font=\sffamily}]
 \node[main node,rounded corners] (1) {1};
 \node[main node,rounded corners] [below = .45cm of 1] (2) {2};
\node[main node,rounded corners] [below right= .5cm and .25cm of 2] (4) {4};
 \node[main node,rounded corners] [below left= .3 cm and .9cm of 2] (3) {3};
 \node[main node,rounded corners] [below right= .3 cm and .9cm of 2] (6) {6};
 \node[main node,rounded corners] [below left= .03cm and .5cm  of 4] (5) {5};
%        \node[main node,rounded corners] [below left=.5cm and .3cm of 1](3) {3};
% \node[main node,rounded corners] [below right=.5cm and .3cm of 1](2) {2};
% \node[main node,rounded corners] [below = .5cm of 3](4) {4};
% \node[main node,rounded corners] [below = .5cm of 2](5) {5};
%
% 
        \path[color=black!20!blue,->,every
        node/.style={font=\sffamily\small}]
   (1) edge node {} (2)
   (1) edge node {} (3)
(1) edge node {} (6)
   (2) edge node {} (4)
   (2) edge node {} (3)
(2) edge node {} (6)
 (3) edge node {} (5)
 (4) edge node {} (5)
 (3) edge [bend left= 998] node {} (6);
 \end{tikzpicture}

\end{minipage}
% another minipage
 \begin{minipage}[t]{0.3\hsize}
\centering
 \begin{tikzpicture}[baseline=(1),->,>=triangle 45, shorten >=0pt,
        auto,thick,
        main node/.style={circle,inner sep=2pt,fill=gray!20,draw,font=\sffamily}]
 	\node[main node,rounded corners] (1) {1};
	 \node[main node,rounded corners] [below = .45cm of 1] (2) {2};
	\node[main node,rounded corners] [below right= .5cm and .25cm of 2] (4) {4};
 	\node[main node,rounded corners] [below left= .3 cm and .9cm of 2] (3) {3};
 	\node[main node,rounded corners] [below right= .3 cm and .9cm of 2] (6) {6};
 	\node[main node,rounded corners] [below left= .03cm and .5cm  of 4] (5) {5};
        \path[color=black!20!red,-,every
        node/.style={font=\sffamily\small}]
   (1) edge node {} (2)
   (1) edge node {} (3)
(1) edge node {} (6)
   (2) edge node {} (4)
   (2) edge node {} (3)
(2) edge node {} (6)
 (3) edge node {} (5)
 (4) edge node {} (5)
(4) edge node {} (3)
 (3) edge [bend left= 998] node {} (6);
 \end{tikzpicture}
\end{minipage}
% last minipage
  \begin{minipage}[t]{0.3\hsize}
\centering
 \begin{tikzpicture}[baseline=(1), ->,>=triangle 45, shorten >=0pt,
        auto,thick,
        main node/.style={circle,inner sep=2pt,fill=gray!20,draw,font=\sffamily}]
        \node[main node,rounded corners] (1) {1};
		\node[main node,rounded corners] [below = 1.8cm of 1] (3) {3};
        \node[main node,rounded corners] [below right= .5cm and .5cm of 1](4) {4};
\node[main node,rounded corners] [below left= .5cm and .5cm of 1](5) {5};
\node[main node,rounded corners] [below = .5cm of 4](6) {6};
\node[main node,rounded corners] [below = .5cm of 5](2) {2};

        \path[color=black!20!red,-,every
        node/.style={font=\sffamily\small}]
(1) edge node {} (4)
(1) edge node {} (5)
(4) edge node {} (6)
(2) edge node {} (5)
(5) edge node{} (6)
;

 \end{tikzpicture}
\end{minipage}
\caption{A graph $G$ (left), $G_{con}$ (middle) and $G_{con}^c$
  (right).  Since $|E_{con}| - |E| = 1<2=d_{con}$, the necessary
  condition in Thm.~\thmssref{our_nec} does not hold.}
 \label{fig:graph7539}
\end{figure}
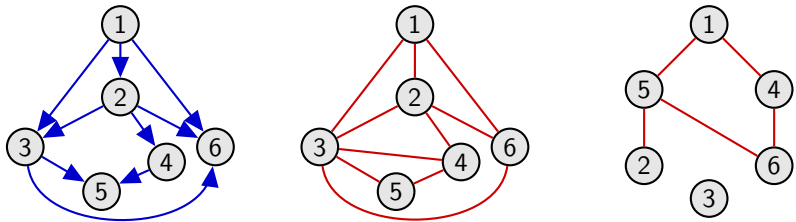

%
%
%
% We also recall the following terminology in graphical theory: \begin{inparaenum} \item A path from a node $v$ to a node $w$ is a sequence of nodes $v = v_0, v_1, \dots, v_k = w$ such that each pair $(v_i, v_{i+1})$ forms an edge in the graph. \item A collider of a path in a DAG is the node $v_i$ for a configuration $v_{i-1} \rightarrow v_i \leftarrow v_{i+1}$ on the path. \item A non-collider of a path in a DAG is the node $v_i$ for a configuration $v_{i-1} \leftarrow v_i \leftarrow v_{i+1}$, $v_{i-1} \rightarrow v_i \rightarrow v_{i+1}$ or $v_{i-1} \leftarrow v_i \rightarrow v_{i+1}$ on the path\end{inparaenum}. Since $L$ corresponds to a common parent of all other observables in $\overline{G}$,  d-separations imply that an edge $v \text{ --- } w \in E_{|L, cov}$ if and only if there exists a path from $v$ to $w$ in $G$ without colliders, and $v \text{ --- } w \in E_{con}$ if and only if there is a path in $G$ from $v$ to $w$ without non-colliders. From \eqref{Matsys}, it is also seen that $\Sigma_{|L} := ( I_m - \Lambda^T)^{-1} \Omega (I_m - \Lambda)^{-1}$ and $\Sigma_{|L}^{-1} := ( I_m - \Lambda) \Omega^{-1} (I_m - \Lambda^T)$ are, respectively, the covariance and concentration matrices of $X$ conditioning on $L$, hence equivalently $ v \text{ --- }w \in E_{|L, cov}$ if and only if $(\Sigma_{|L})_{vw}$ is not identically zero, and analogously for  $E_{ con}$ and $\Sigma_{|L}^{-1}$.

In addition to the closely related work of \cite{stang1997} and
\cite{vicard2000}, identifiability of directed Gaussian models with
one latent variable has been studied by \cite{wermuth2005}.  The
models we treat here
are special cases with the latent node being a common parent of all
the observable nodes.  As we review in more detail in
\secref{prior_work}, we can readily adapt the sufficient graphical
criteria given in \cite{wermuth2005} for certifying that the model
$\mathcal{N}_*(G)$ of a given DAG $G$ is generically finitely
identifiable with respect to \defref{gfiden}.  Our own sufficient
condition stated in \thmref{our_suff} is stronger, in the sense that
every DAG $G$ satisfying the sufficient conditions in
\cite{wermuth2005} necessarily satisfies the condition in
\thmref{our_suff}.  However, when it applies the result of
\cite{wermuth2005} yields a stronger conclusion than our generic
finiteness result.  Indeed as we also emphasize in the discussion in
\secref{discussion}, their conditions imply that the parmetrization is
generically 2-to-1.

We will prove the above stated \thmsref{our_suff} and
\thmssref{our_nec} in \secref{suff_and_nec}.  Since the
parametrization map in \eqref{phi} is polynomial, the generic finite
identifiability of a given model is decidable by algebraic techniques
that involve Gr\"obner basis computations. In \secref{example}, we
will study the applicability of our graphical criteria via such algebraic
computations for all models $\mathcal{N}_*(G)$ of DAGs $G$ with
$m = 4, 5, 6$ nodes.  \secref{hara} will give results on situations
where we can determine generic finite identifiability of a model
$\mathcal{N}_*(G)$ based on knowledge about the generic finite
identifiability of a model $\mathcal{N}_*(G')$, where $G'$ is an
induced subgraph of $G$.

Before ending this introduction, however, we comment on the role that
Markov equivalence plays in our problem. Recall that two DAGs defined
on the same set of nodes are Markov equivalent if they have the same
d-separation relations. The following theorem, which will be proved in Appendix
\ref{sec:proofs}, says that generic finite identifiability is a
property of Markov equivalence classes of DAGs.

\begin{theorem}\label{thm:Markov_equiv}
Suppose $G_1 = (V, E_1)$ and $G_2 = (V, E_2)$ are two Markov equivalent DAGs on the same set of nodes $V$. Then the model $\mathcal{N}_*(G_1)$ is generically finitely identifiable if and only if the same is true for $\mathcal{N}_*(G_2)$.
\end{theorem}

%Following other previous work \cite{stang1997, vicard2000}, \cite{wermuth2005} gave sufficient graphical conditions for generic finite identifiability of a $G^*$ model. In fact, their results are sufficient for generic two-to-one identifiability, i.e., graphs satisfying their criteria give models that are generically identifiable up to a sign change in $\delta$. Section 2 will give a review of their work. In Section 3, we gives a sufficient condition for generic finite identifiability that is stronger than the one in \cite{wermuth2005}, although generic two-to-one identifiability is not guaranteed. Another contribution of our work is a necessary condition for a $G^*$ model to be generically finitely identifiable. In particular, it draws on some observations made in \cite{vicard2000} that can be further explored. Section 4 discusses results on situations where we can determine generic finite identifiability of a $G^*$ model based on knowledge about the generic finite identifiability of a $G'^*$ model, where $G'$ is an induced subgraph of $G$. 
%

%%% SECTION 2%%%%

\section{Prior work} \label{sec:prior_work}

\cite{wermuth2005} give sufficient graphical conditions for
identifiability of directed Gaussian graphical models with one latent
variable that can be any node in the DAG.  We revisit their result in
the context of the models from \defref{G*mod} and formulate it in
terms of generic finite identifiability.  (As was mentioned in the
Introduction, their result yields in fact the stronger conclusion of a
generically 2-to-1 parametrization.)  We begin by stating a well-known
fact about DAG models without latent variables.

\begin{lemma}\label{lem:obs_ident}
  For any DAG $G=(V,E)$ with $m=|V|$ nodes, the map
  \[
  (\Lambda,\Omega) \;\mapsto\; ( I_m - \Lambda^T)^{-1} \Omega (I_m -
  \Lambda)^{-1} 
  \]
  is injective on the domain $\mathbb{R}_E\times  \diag_m^+$ and has a
  rational inverse.
\end{lemma}

\begin{proof}
  For any $(\Lambda,\Omega)\in \mathbb{R}_E\times
  \diag_m^+$, let $\Sigma = (\sigma_{vw}) =( I_m - \Lambda^T)^{-1}
  \Omega (I_m - \Lambda)^{-1}$.
  % Consider the Cholesky decomposition $\Sigma = R^TR$, where
  % $R$ is upper triangular.  Then the entries $\lambda_{vw}$ and
  % $\omega_v$ of $\Lambda$ and $\Omega$, respectively, can be solved
  % for by considering the matrix equation
  % $R = \Omega^{1/2} (I_m - \Lambda)^{-1}$ column by column from left
  % to right, where
  % $\Omega^{1/2} = \diag(\sqrt{\omega_1}, \dots, \sqrt{\omega_m})$.
  % The entries of $R$ are rational functions of $\Sigma$ and so are
  % $\Lambda$ and $\Omega$.  
  Let $pa(v) = \{w : w \rightarrow v \in E\}$ be the parent set of the
  node $v$.  Then one can show, by induction on $m$ and considering a
  topological ordering of $V$, that
  \[
  \Lambda_{pa(v),v} \;=\; \left(\Sigma_{pa(v),pa(v)}\right)^{-1} \Sigma_{pa(v),v}
  \]
  and
  \[
  \Omega_{vv} \;=\;
  \sigma_{vv}-\Sigma_{v,pa(v)}\left(\Sigma_{pa(v),pa(v)}\right)^{-1}
  \Sigma_{pa(v),v};
  \]
  compare, for instance, \citet[\S8]{richardson:ancestral:2002}.
\end{proof}

Let the random vector $X$ and the latent variable $L$ have their joint
distribution specified via the equation system from~\eqref{Matsys}.
Write $\Sigma_{|L}$ for the conditional covariance matrix of $X$ given
$L$.  Then it holds that
\begin{equation}
  \label{eq:sigmaL}
  \Sigma_{|L} = ( I_m - \Lambda^T)^{-1} \Omega (I_m - \Lambda)^{-1}.
\end{equation}
Hence, by Lemma~\ref{lem:obs_ident}, when knowing $\Sigma_{|L}$ we can
uniquely solve for the pair $(\Lambda, \Omega)$, which are rational
functions of $\Sigma_{|L}$.  Writing $\Sigma$ for the (unconditional)
covariance matrix of $X$, we have from~(\ref{sigmacov}) that
\[
\Sigma_{|L} = \Sigma - (I_m - \Lambda^T)^{-1} \delta \delta^T (I_m -
\Lambda)^{-1}. 
\]
Consequently, $(\Lambda,\Omega)$ can be recovered uniquely from
$\Sigma$ and $(I_m - \Lambda^T)^{-1} \delta$.  The results of
\cite{wermuth2005} then address identification of the vector
$(I_m - \Lambda^T)^{-1} \delta$, which holds the covariances between
each coordinate of $X$ and the latent variable $L$.  We obtain the
following observation.

\begin{proposition}[Adapted from Stanghellini and Wermuth,
  2005]\label{prop:wermuth2005}
Let $G = (V, E)$ be a DAG. The model $\mathcal{N}_*(G)$ is generically
finitely identifiable if 
\begin{enumerate}
\item every connected component of $G_{|L, cov}^c = (V, E_{|L, cov}^c)$ has an odd cycle, or
\item every connected component of $G_{con}^c = (V, E_{con}^c)$ has an odd cycle.
\end{enumerate}
\end{proposition}

\begin{proof}
  Theorem $1$ in \cite{wermuth2005} gives $(i)$ or $(ii)$ as a
  sufficient condition for identifying, up to sign, the $m$-vector
  $(I_m - \Lambda^T)^{-1} \delta$ when
  $\Sigma = \phi_G (\Lambda, \Omega, \delta)$ for a generic point
  $(\Lambda, \Omega, \delta)$ in $\Theta$.  In this case, we can
  uniquely recover the conditional covariance matrix $\Sigma_{|L}$
  from~(\ref{eq:sigmaL}) and also the pair $(\Lambda,\Omega)$ by
  \lemref{obs_ident}. After identifying $\Lambda$, $\delta$ can be
  solved for, up to sign, by the previous knowledge of
  $(I_m - \Lambda^T)^{-1} \delta$.  Hence, $(i)$ or $(ii)$ is in fact
  a sufficient condition for generic finite identifiability of
  $\mathcal{N}_*(G)$.
\end{proof}

\figref{SWDAG} shows a DAG with $m = 5$ nodes that satisfies the condition
of \propref{wermuth2005}$(ii)$.

\begin{figure}[t] 
\centering
\ \

\ \ 
 \begin{minipage}[t]{0.3\hsize}
\centering
 \begin{tikzpicture}[->,>=triangle 45, shorten >=0pt,
        auto,thick,
        main node/.style={circle,inner sep=2pt,fill=gray!20,draw,font=\sffamily}]

        \node[main node,rounded corners] (1) {1};
        \node[main node,rounded corners] [below left=.5cm and .3cm of 1](3) {3};
 \node[main node,rounded corners] [below right=.5cm and .3cm of 1](2) {2};
 \node[main node,rounded corners] [below = .5cm of 3](4) {4};
 \node[main node,rounded corners] [below = .5cm of 2](5) {5};

        \path[color=black!20!blue,->,every
        node/.style={font=\sffamily\small}]
    (1) edge node {} (2)
        (1) edge node {} (3)
        (2)  edge node {} (3)
       (3) edge node {} (4) 
        (2) edge node {} (5);
 \end{tikzpicture}

\end{minipage}
% another minipage
 \begin{minipage}[t]{0.3\hsize}
\centering
 \begin{tikzpicture}[->,>=triangle 45, shorten >=0pt,
        auto,thick,
        main node/.style={circle,inner sep=2pt,fill=gray!20,draw,font=\sffamily}]

        \node[main node,rounded corners] (1) {1};
        \node[main node,rounded corners] [below left=.5cm and .3cm of 1](3) {3};
 \node[main node,rounded corners] [below right=.5cm and .3cm of 1](2) {2};
 \node[main node,rounded corners] [below = .5cm of 3](4) {4};
 \node[main node,rounded corners] [below = .5cm of 2](5) {5};

        \path[color=black!20!red,-,every
        node/.style={font=\sffamily\small}]
    (1) edge node {} (2)
        (1) edge node {} (3)
        (2)  edge node {} (3)
       (3) edge node {} (4) 
        (2) edge node {} (5);
 \end{tikzpicture}

\end{minipage}
% last minipage
  \begin{minipage}[t]{0.3\hsize}
\centering
 \begin{tikzpicture}[->,>=triangle 45, shorten >=0pt,
        auto,thick,
        main node/.style={circle,inner sep=2pt,fill=gray!20,draw,font=\sffamily}]

        \node[main node,rounded corners] (1) {1};
        \node[main node,rounded corners] [below left= 1cm and .2cm of 1](4) {4};
 \node[main node,rounded corners] [below right=1cm and .2cm of 1](5) {5};
 \node[main node,rounded corners] [left = .4cm of 1](2) {2};
 \node[main node,rounded corners] [right = .4cm of 1](3) {3};

        \path[color=black!20!red,-,every
        node/.style={font=\sffamily\small}]
    (1) edge node {} (4)
        (1) edge node {} (5)
        (4)  edge node {} (5)
       (2) edge node {} (4) 
        (3) edge node {} (5);
 \end{tikzpicture}
\end{minipage}
\caption{A graph $G$ (left) satisfying the sufficient condition in
  \propref{wermuth2005}, with $G_{con}$ (middle) and
  $G_{con}^c$ (right).}
 \label{fig:SWDAG}
\end{figure}

We conclude this review of prior work by pointing out that any model
$\mathcal{N}_* (G)$ that can be determined to be generically finitely
identifiable using \propref{wermuth2005} can also be found to have
this property using our new \thmref{our_suff}.

\begin{proposition}\label{prop:stronger}
  A DAG $G=(V,E)$ satisfying either one of the conditions in
  \propref{wermuth2005} necessarily satisfies the condition in
  \thmref{our_suff}.
\end{proposition}

\begin{proof}
  Let $G_{|L, cov} = (V, E_{|L, cov})$ and $G_{con} = (V, E_{con})$.
  An edge $v \rightarrow w \in E$ also present itself as an undirected
  edge in both $E_{|L, cov}$ and $E_{con}$.  Hence, when ignoring the
  directionality of its edges, $G$ is a subgraph of both $G_{|L, cov}$
  and $G_{con}$ and, thus, $G^c$ is a supergraph of both
  $G_{|L, cov}^c$ and $G_{con}^c$. As such, if every connected
  component of $G_{|L, cov}^c$, or of $G_{con}^c$, contains an odd cycle,
  the same is true of $G^c$.
\end{proof}

\section{Criteria based on the Jacobian of  parametrization maps} \label{sec:suff_and_nec}

In this section, we prove \thmsref{our_suff} and \thmssref{our_nec}.
Let $G = (V, E)$ be a fixed DAG with $m=|V|$ nodes, and let
$\Theta := \mathbb{R}_E \times \diag_m^+ \times \mathbb{R}^m$ denote
again the domain of the parametrization 
\begin{align*} %\label{phi}
\phi_{G}  :
(\Lambda, \Omega, \delta) &\longmapsto ( I_m - \Lambda^T)^{-1} (\Omega + \delta \delta^T)(I_m - \Lambda)^{-1}
\end{align*}
of the covariance matrix of the distributions in model
$\mathcal{N}_*(G)$.  We begin by introducing other mappings that are
generically finite-to-one if and only if $\phi_G$ is generically
finite-to-one.

First, it will be helpful to study the map
\begin{equation}\label{phi_tilde}
\tilde{\phi}_{G} : (\Lambda, \Omega, \delta) \longmapsto ( I_m - \Lambda^T)^{-1} \Omega (I_m - \Lambda)^{-1} + \delta \delta^T,
\end{equation}
defined on $\Theta$.  Second, focusing on concentration instead of
covariance matrices, we will also consider the maps
\begin{align}
 \label{varphi}
  \varphi_{G} :  (\Lambda, \Psi, \gamma) 
  &\longmapsto ( I_m - \Lambda) (\Psi- \gamma \gamma^T)(I_m - \Lambda^T),\\
  \label{varphi_tilde}
\tilde{\varphi}_G : (\Lambda, \Psi, \gamma) 
  &\longmapsto (I_m - \Lambda) \Psi(I_m - \Lambda^T) - \gamma \gamma^T.
\end{align}

\begin{lemma}
  \label{lem:equiv_maps}
  The parametrization $\phi_G$ is generically finite-to-one if and
  only if any one of the maps $\tilde{\phi}_{G}$, $\varphi_{G}$ and
  $\tilde{\varphi}_G$ is generically finite-to-one.
\end{lemma}
\begin{proof}
  Consider first the map $\tilde{\phi}_G$ for which it holds that
  $\phi_G = \tilde{\phi}_G \circ g$, where
  \[
  g : (\Lambda, \Omega, \delta) \longmapsto (\Lambda, \Omega, ( I_m -
  \Lambda^T)^{-1}\delta)
  \]
  is a diffeomorphism that maps $\Theta$ to itself.  By the chain
  rule, the Jacobian of $\phi_G$ at $(\Lambda, \Omega, \delta)$ is the
  product of the Jacobian of $\tilde{\phi}_G$ at
  $g(\Lambda, \Omega, \delta)$ and the Jacobian of $g$ at
  $(\Lambda,\Omega,\delta)$.  Now the latter matrix is invertible on
  all of $\Theta$ since $g$ is a diffeomorphism.  It follows that
  there exists a point in $\Theta$ at which the Jacobian of $\phi_G$
  has full column rank if and only if the same is true for
  $\tilde{\phi}_G$.  For the Jacobian of a polynomial map such as
  $\phi_G$ and $\tilde{\phi}_G$, full column rank at a single point
  implies generically full column rank; use the subdeterminants that
  characterize a drop in rank to define a proper algebraic subset of
  exceptions, see also \citet[Lemma 9]{geiger:2001}.  The claim about
  $\phi_G$ and $\tilde{\phi}_G$ follows from
  \lemref{localequalfinite}.

  Let $h : (\Lambda, \Psi, \gamma) \longmapsto (\Lambda, \Psi, (I_m - \Lambda)\gamma)$. Since $\varphi_G =\tilde{\varphi}_G  \circ  h$, by the same argument as above it also holds
  that $\varphi_G$ is generically finite-to-one if and only if
  $\tilde{\phi}_G$ has this property.

  In order to complete the proof of the lemma it suffices to show that
  $\phi_G$ is generically finite-to-one if and only if the same holds
  for $\varphi_G$.  Define another diffeomorphism from $\Theta$ to
  itself as
  \[
  \rho :(\Lambda, \Omega, \delta)\longmapsto \left(\Lambda, \Omega^{-1},
  (1+\delta^T\Omega^{-1}\delta)^{-1/2} \Omega^{-1}\delta\right).
  \]
  Writing $inv$ for matrix inversion, we then have that 
  \begin{equation}
    \label{eq:varphi:claim}
    \ inv \circ \phi_G =\varphi_G \circ \rho
  \end{equation}
  because of the identity $(\Omega + \delta \delta^T)^{-1} = (\Psi- \gamma \gamma^T)$
  with $\Psi=\Omega^{-1}$ and $\gamma = k^{-1/2} \Psi\delta$, where
  $k = 1 + \delta^T \Psi \delta>0$; see e.g.~\citet[p.~33]{raolinear}.
  Using (\ref{eq:varphi:claim}), 
  the equivalence of being generically finite-to-one for $\phi_G$ and
  $\varphi_G$ may be argued similarly as for the maps considered
  earlier.    
\end{proof}

Let $J(\tilde{\varphi}_G)$ be the Jacobian matrix of the map
$\tilde{\varphi}_G$ from \eqref{varphi_tilde}. It will be examined to
prove \thmref{our_suff}.  In light of \lemsref{localequalfinite} and
\lemssref{equiv_maps}, we will show that if $G$ satisfies the
condition in \thmref{our_suff}, then $J(\tilde{\varphi}_G)$ is
generically of full column rank, implying that $\phi_G$ is generically
finite-to-one.  Our arguments will make use of the following lemma
that rests on observations made in \cite{vicard2000}.

\begin{lemma}\label{lem:lem_vicard}
  Let $G = (V, E)$ be an undirected graph, and let
  $f_G:\mathbb{R}^V\to\mathbb{R}^E$ be the map with coordinate
  functions
  \[
  f_{G,vw}(x) \;=\; x_vx_w, \quad v \text{ --- } w \in E.
  \]
  Then the Jacobian of $f_G$ has generic rank $m-d$, where $m=|V|$ is
  the number of nodes and $d$ is the number of connected components of
  $G$ that do not contain an odd cycle.
\end{lemma}

\begin{proof}
  For simpler notation, let $f:=f_G$.  Let $J_f$ be the Jacobian
  matrix of the polynomial map $f$, and let $\text{ker}(J_f)$ be its
  kernel.  By the rank theorem \citep[p.~229]{babyrudin}, the
  dimension of $\text{ker}(J_f)$ is generically equal to the dimension
  of the fiber $\mathcal{F}_f$; recall~(\ref{eq:fiber:notation}).
  Since $\text{rank} (J_f) = m - \dim(\text{ker}(J_f))$, it suffices
  to show that $\mathcal{F}_f$ has generic dimension $d$.
  
  Since the claim is about a generic property, we may restrict the
  domain of $f$ to the open set
  $\mathcal{X} := (\mathbb{R}\setminus\{0\})^m$.  This assumption is
  made so that Lemma $1$ in \cite{vicard2000} is applicable later
  without difficulty.  Now, fix a point
  $y \in f(\mathcal{X})\subset\mathbb{R}^E$.  The elements of the
  fiber $\mathcal{F}_f(y)$ are the vectors $x\in\mathbb{R}^m$, or
  equivalently, $x\in\mathcal{X}$, that are solutions to the system of
  equations
  \begin{equation} \label{f_sys}
    y_{vw} = x_v x_w, \quad v\text{ --- } w\in E.
  \end{equation} 

  Let $G_1= (V_1, E_1), \dots, G_k= (V_k, E_k)$ be the connected
  components of $G$, so that $V_1,\dots,V_k$ form a partition of $V$
  and $E_1,\dots,E_k$ partition $E$.  Let $k' \leq k$ be the number of
  connected components containing two nodes at least. Without loss
  of generality, assume $G_{k'+1}, \dots, G_k$ are all the
  connected components with only a single node.  Then the equations
  listed in \eqref{f_sys} can be arranged to form $k'$ disjoint
  subsystems indexed by $i=1, \dots, k'$.  The $i$-th subsystem has the
  form
  \begin{equation} \label{f_subsys}
    y_{vw} = x_v x_w, \quad v\text{ --- }w \in E_i
  \end{equation} 
  and exclusively involves the variables $\{x_v : v \in V_i\}$.  By
  Lemma $1$ in \cite{vicard2000} and also the relevant discussion in
  the proof of Theorem $1$ in the same paper, the solution set
  to~(\ref{f_subsys}) either contains two points or can be
  parametrized by a single free variable in $\mathbb{R}$.  The former
  case arises if and only if $G_i $ contains an odd cycle.  It follows
  that the dimension of the solution set of~(\ref{f_subsys}) is zero
  when $G_i$ contains an odd cycle, and it has dimension one if $G_i$
  does not contain an odd cycle.  In addition, each singleton
  component $G_i = (V_i, \emptyset)$ for $i = k'+1, \dots, k$ provides
  one additional dimension to the fiber $\mathcal{F}_f(y)$, since the
  corresponding variables in $x$ are not restricted by any equations.
  We conclude that the dimension of $\mathcal{F}_f(y)$ equals the
  number of connected components $G_i$ that do not contain an odd
  cycle.
\end{proof}

We return to the object of study, namely, the map $\tilde{\varphi}_G$
which sends the $(2m + |E|)$-dimensional set
$\Theta = \mathbb{R}_E \times \diag^+_m \times \mathbb{R}^m$ to the
${m +1 \choose 2}$-dimensional space of symmetric $m \times m$
matrices.  The Jacobian $J(\tilde{\varphi}_G)$ is of size
${m +1 \choose 2} \times (2m + |E|)$, and we index its rows by pairs
$(v,w)$ with $1\le v< w\le m$, whereas in Section~\ref{sec:intro} we
assume the vertex set $V=\{1,\dots,m\}$ to be topologically ordered.
We now describe a particular way of arranging the rows and columns of
$J(\tilde{\varphi}_G)$.

Define the set of ``non-edges'' as
$N := \{ (v, w): v <w \text{ and } (v, w) \not \in E \}$; we will also
write $v \not \rightarrow w$ to express that $(v, w) \in N$.  Also,
define $D := \{(v, v) : v \in V\}$, so that $D \cup E \cup N$ index
all entries in the upper triangular half of an $m \times m$ symmetric
matrix.  The rows of $J(\tilde{\varphi}_G)$ are now arranged in the
order $D$, $E$ and $N$.  The columns of $J(\tilde{\varphi}_G)$ are
indexed such that partial derivatives with respect to the free input
variables in the triple $(\Lambda, \Psi, \gamma)$ appear from left to
right, in the order $\Psi$, $\Lambda$ and $\gamma$.  In other words,
we partition $J(\tilde{\varphi}_G)$ into 9 blocks as follows:
\begin{equation} \label{Jacob}
J(\tilde{\varphi}_G) = \kbordermatrix{\mbox{}&\Psi&\Lambda&\gamma\\
D&\cdots&\cdots&\cdots\\
E&\cdots&\cdots&\cdots\\
N&\cdots&\cdots&\cdots\\
}.\end{equation}

The following lemma is obtained by inspection of the partial
derivatives of $\tilde{\varphi}_G$.  Its proof appears in
Appendix \ref{sec:proofs}.

\begin{lemma} \label{lem:Jacob_structure} The Jacobian matrix
  $J(\tilde{\varphi}_G)$ is generically of full column rank provided
  that the submatrix $[J(\tilde{\varphi}_G)]_{N, \gamma}$ is so.
\end{lemma}

We now give the proof of \thmref{our_suff}.

\begin{proof}[Proof of \thmref{our_suff}]
By \lemsref{localequalfinite} and \lemssref{Jacob_structure} , it
suffices to show that $[J(\tilde{\varphi}_G)]_{N, \gamma} $ is
generically of full column rank. For each $v \not \rightarrow  w \in
N$,  
\begin{equation} \label{coor_maps}
  \left[\tilde{\varphi}_G (\Lambda, \Psi, \gamma) \right]_{vw}
  =\left[(I_m  - \Lambda) \Psi (I_m  - \Lambda^T)\right]_{vw} -
  \gamma_v \gamma_w. 
\end{equation}
Note that only the right most term in~(\ref{coor_maps}) contributes to
the partial derivatives of $\tilde{\varphi}_G$ with respect to
$\gamma = (\gamma_v)_{v\in \{1, \dots, m\}}$.  

Ignoring the directionality of non-edges in $N$, define the undirected
graph $H = (V, N)$ to which we associate a map $f_H$ as in
\lemref{lem_vicard}.  Then
\[ [J(\tilde{\varphi}_G)]_{N, \gamma} = -J_{f_H}.
\]
But $J_{f_H}$ has generically full column rank by \lemref{lem_vicard}
because, in fact, $H$ is equal to the complementary graph $G^c$ for
which we assume that all connected components contain an odd cycle.
\end{proof}

We remark that \thmref{our_suff} can also be proven by studying the
Jacobian of the map $\tilde{\phi}_G$ from \eqref{phi_tilde}.  We chose
to work with $\tilde{\varphi}_G$ above since this allowed us to avoid
consideration of the inverse of the matrix $I_m - \Lambda$.  For
\thmref{our_nec}, however, we consider both $\tilde{\varphi}_G$ and
$\tilde{\phi}_G$.

\begin{proof}[Proof of \thmref{our_nec}]
  We first prove the necessity of condition $(i)$ by showing that if
  $|E_{con}| - |E| < d_{con}$, then the Jacobian matrix
  $J(\tilde{\varphi}_G)$ always has row rank less than $2m+ |E|$.
  This implies that it cannot be of full column rank which implies the
  failure of generic finite identifiability by
  \lemref{localequalfinite}.

  As in the proof of \thmref{our_suff}, we consider the set of
  non-edges $N$, which we now partition as $N = N_1\dot\cup N_2$,
  where $N_1=\{v \not \rightarrow w \in E : \text{
    $v \text{ --- }w \in E_{con}$}\}$, and $N_2 = N \setminus N_1$.
  Accordingly, we can partition the submatrix
  $[J(\tilde{\varphi}_G)]_{N,\{\Psi, \Lambda, \gamma\}}$ into two
  block of rows indexed by $N_1$ and $N_2 $ as
  \begin{equation}
    \label{eq:our-nec-proof-zeros}
    [J(\tilde{\varphi}_G)]_{N,\{\Psi, \Lambda, \gamma\}}
    =\kbordermatrix{\mbox{}&\Psi&\Lambda&\gamma\\
      N_1&\cdots&\cdots&\cdots\\
      N_2& 0 & 0 &\cdots\\
    }.
  \end{equation}
  To see that the submatrix
  $[J(\tilde{\varphi}_G)]_{N_2, \{ \Psi, \Lambda \}} = 0$, observe
  first that an entry of $(I - \Lambda) \Psi(I - \Lambda^T)$ is the
  zero polynomial if and only if the same is true for
  $\Sigma_{|L}^{-1}$, where $\Sigma_{|L}$ is the matrix
  from~(\ref{eq:sigmaL}).  Second, by definition of $E_{con}$ and
  $N_2$, if $(v,w) \in N_2$ then $(\Sigma_{|L}^{-1})_{vw}= 0$.

  Next, observe that to prove the necessity of condition $(i)$ it
  suffices to show that the rank of
  $[J(\tilde{\varphi}_G)]_{N_2, \gamma} $ cannot be larger than
  $m - d_{con}$.  Indeed, if this is true, then there exists a subset
  $N_2' \subset N_2$ with $|N_2'| = m - d_{con}$, such that the
  submatrix
  \[ 
  [J(\tilde{\varphi}_G)]_{\{D, E, N_1, N_2'\},\{\Psi, \Lambda,
    \gamma\}}
  \]
  has the same rank as the original Jacobian matrix
  $J(\tilde{\varphi}_G)$.  However, the submatrix
  $[J(\tilde{\varphi}_G)]_{\{D, E, N_1, N_2'\},\{\Psi, \Lambda,
    \gamma\}}$
  has $2m+ |E_{con}| - d_{con}$ rows, and thus its rank is less than
  $2m+ |E|$ because under condition $(i)$ we have
  $|E_{con}| - |E| < d_{con}$.  As a result, $J(\tilde{\varphi}_G)$
  cannot be of full column rank.

  It now remains to show that $[J(\tilde{\varphi}_G)]_{N_2, \gamma}$
  has rank at most $m - d_{con}$.  Observe that the undirected graph
  $(V,N_2)$ is equal to the complementary graph $(G_{con})^c$.
  Moreover, $[J(\tilde{\varphi}_G)]_{N_2, \gamma}$ is equal to the
  negative Jacobian of the map $f_{(G_{con})^c}$ that we get by applying the
  construction from \lemref{lem_vicard} to $(G_{con})^c$; recall the
  proof of \thmref{our_suff}.  Applying \lemref{lem_vicard}, we find
  that $[J(\tilde{\varphi}_G)]_{N_2, \gamma}$ has generic rank
  $m - d_{con}$, which is also the maximal rank that
  $[J(\tilde{\varphi}_G)]_{N_2, \gamma}$ may have.

  The proof of $(ii)$ follows the exact same argument as that of
  $(i)$, by replacing \begin{inparaenum}[(a)] \item $G_{con}$ with
    $G_{|L, cov}$, \item $d_{con}$ with $d_{cov}$, \item
    $\tilde{\varphi}_G$ with $\tilde{\phi}_G$, \item$\Psi$ with
    $\Omega$, \item$\gamma$ with $\delta$ and \item
    $J(\tilde{\varphi}_G)$ with $J(\tilde{\phi}_G)$\end{inparaenum},
  where $J(\tilde{\phi}_G)$ is partitioned as
  \begin{equation} \label{Jacob2}
    J(\tilde{\phi}_G) = \kbordermatrix{\mbox{}&\Omega&\Lambda&\delta\\
      D&\cdots&\cdots&\cdots\\
      E&\cdots&\cdots&\cdots\\
      N&\cdots&\cdots&\cdots\\
    },\end{equation}
  similarly to \eqref{Jacob}.
\end{proof}

\section{Algebraic computations and examples}\label{sec:example}

As explained in \citet[\S3]{drton:prague} and \cite{garcia:2010},
identifiability properties of a model such as $\mathcal{N}_*(G)$ can
be decided using Gr\"obner basis techniques from computational
algebraic geometry \citep{coxlittleoshea}.  While these techniques are
tractable only for small to moderate size problems, we were able to
perform an exhaustive algebraic study of all DAGs $G=(V,E)$ with
$m \le 6$ nodes.  Beyond a mere decision on whether the
parametrization map $\phi_G$ is generically 1-to-1, the algebraic
methods also provide information about the generic cardinality of the
fibers of $\phi_G$ as a map defined on complex space.
%%  $\mathbb{C}^{2m+|E|}$.

% Let $\mathbb{C}_E$ be similiarly
% defined as $\mathbb{R}_E$ whose members are allowed to take values in
% $\mathbb{C}$, and $\diag_m^\mathbb{C}$ be the set of all $m \times m$
% diagonal complex matrices.

\begin{definition}
  \label{def:gen-k-1}
  For a DAG $G = (V, E)$, let $\phi_G^\mathbb{C}$ be the map obtained
  by extending $\phi_G$ to the complex domain $\mathbb{C}^{2m+|E|}$.
  % $\Theta^\mathbb{C} := \mathbb{C}_E \times \diag_m^\mathbb{C} \times
  % \mathbb{C}^m$.
  If the (complex) fibers of $\phi_G^\mathbb{C}$ are generically of
  cardinality $k$, then we say that $\phi_G^\mathbb{C}$ is generically
  $k$-to-one.
\end{definition}

The language of \defref{gen-k-1} allows us to give a refined
classification of DAGs $G$ in terms of the identifiability properties
of the parametrization of model $\mathcal{N}_*(G)$.  Indeed,
$\mathcal{N}_*(G)$ is generically finitely identifiable if and only if
$\phi_G^\mathbb{C}$ is generically $k$-to-one for some $k < \infty$.

\begin{remark}
  The generic size of the fibers of $\phi_G^\mathbb{C}$ equals the
  generic size of the fibers of the complex extensions of the three
  maps from \lemref{equiv_maps}.  The map $\tilde\varphi_G$ has low
  degree coordinates and tends to be the easiest to work with in
  algebraic computation.  Another approach that can be useful is to
  adapt the algorithm described in Section 8 of the supplementary
  material for \cite{foygel2012}.  To do this note that for
  $\Lambda\in\mathbb{C}^E$ there exist complex choices of $\Omega$ and
  $\delta$ such that $\phi_G(\Lambda,\Omega,\delta)=\Sigma$ if and
  only if $(I-\Lambda^T)\Sigma(I-\Lambda)$ is a matrix that is the sum
  of a diagonal matrix, namely, $\Omega$, and a symmetric matrix of
  rank 1, namely, $\delta\delta^T$.  Whether a matrix is of the latter
  type can be tested using tetrads, that is, $2\times 2$
  subdeterminants involving only off-diagonal entries of the matrix;
  see also~(\ref{tetrad_eq}) below.  The tetrads of a matrix form a
  Gr\"obner basis \citep{Loera,tedandpen}.
\end{remark}

\begin{table}[t]
\centering
\caption{Counts of unlabeled DAGs $G$ with $m$ nodes, at most ${m +1
    \choose 2} - 2m$ edges, and complex parametrization
  $\phi_G^\mathbb{C}$ generically
  $k$-to-one.  Counts are also given for DAGs that satisfy the sufficient
  conditions from Thm.~\thmssref{our_suff} and
  Prop.~\propssref{wermuth2005}, and 
  DAGs that fail to satisfy
  the necessary condition from Thm.~\thmssref{our_nec}.
}
\begin{tabular}{c|r@{\hspace{0.5cm}}r@{\hspace{.5cm}}r}

$m$ & 4 &5& 6\\ 
\hline
\hline
$k < \infty$& 5 & 95& 3344 \\ 
$k = 2$ & 5 & 87& 2961 \\ 

$k = 4$ & 0 & 8& 345 \\ 
$k = 6$& 0 & 0& 24 \\ 

$k = 8$& 0 & 0 &  14 \\[0.1cm]
%$\checkmark$ 
Prop.~\propssref{wermuth2005} & 5 & 49& 985 \\ 
%$\checkmark$ 
Thm.~\thmssref{our_suff}& 5 & 88& 2957 \\[0.05cm] 
% $\mathrm{Suff}_{\mathrm{SW}}$ & 5 & 49& 985 \\ 
% %$\checkmark$ 
% S& 5 & 88& 2957 \\ 
\hline
$k = \infty$ & 1 & 20 &  552 \\[0.025cm]

%\text{\sffamily X} 
Thm.~\thmssref{our_nec} & 1 & 20 &  361 \\[0.05cm] 
%% N & 1 & 20 &  361 \\ 
\hline  \hline
Total \# of DAGs & 6 & 115 &  3896 \\ 
% $\checkmark$  N & 5 & 95 &  3535 \\ 
%\hline
% \multicolumn{4}{c}{}\\
% \multicolumn{4}{l}{\textsuperscript{*}\footnotesize{S = {\bf S}ufficient condition in Thm.~\thmssref{our_suff} }}\\
% \multicolumn{4}{l}{\textsuperscript{*}\footnotesize{N  = {\bf N}ecessary condition in Thm.~\thmssref{our_nec} }}\\
% \multicolumn{4}{l}{\textsuperscript{*}\footnotesize{SW  =  {\bf S}tanghellini-{\bf W}ermuth's condition (Prop.~\propssref{wermuth2005})}}\\
% % \multicolumn{4}{l}{\textsuperscript{*}\footnotesize{$\checkmark = \text{satisfied}$, $\text{\sffamily X}  = \text{not satisfied}$}}\\
\end{tabular}
\label{tab:experiment}
\end{table}

\tabref{experiment} lists out the counts of DAGs $G = (V, E)$, with
$4\le m\le 6$ nodes, that have $\phi_G^\mathbb{C}$ generically
$k$-to-one, for all possible values of $k$.  The table also gives the
the counts of DAGs satisfying the conditions in \thmsref{our_suff} and
\thmssref{our_nec} as well as \propref{wermuth2005}.  DAGs with
${m +1\choose 2} - 2m < |E|$, which trivially give generically
$\infty$-to-one maps $\phi_G^{\mathbb{C}}$ in view of
\corref{basicthm}, are excluded.  We emphasize that the counts are
with respect to unlabeled DAGs, that is, all DAGs that are isomorphic
with respect to relabeling of nodes are counted as one unlabeled graph.

In the considered settings the condition in
\thmref{our_suff} is very successful in certifying DAGs with a
generically finitely identifiable model. For instance, when $m = 6$,
it is able to correctly identify $2957$ out of $3344$ such graphs.
The previously known sufficient condition of \cite{wermuth2005}
identifies $985$ of them.  Our necessary condition in \thmref{our_nec}
is also useful in assessing graphs that give generically
infinite-to-one models.  For instance, when $m = 6$, we find that
$361$ of $552$ such graphs violate the
condition; recall the example from \figref{graph7539}. 

While, by \propref{stronger}, our sufficient condition in
\thmref{our_suff} is stronger than that in \propref{wermuth2005} for
generic finite identifiability, the latter condition, due to
\cite{wermuth2005}, in fact implies that $\phi_G^\mathbb{C}$ is
generically $2$-to-one.  For $m=5$, there are $6$ DAGs that satisfy
the condition in \thmref{our_suff} but give generically $4$-to-one
maps $\phi_G^\mathbb{C}$.  The graph from \figref{graph675} is an
example.  We note that for this graph $G$ the fibers of
$\phi_G^\mathbb{C}$ intersect the statistically relevant set $\Theta$
in either 2 or 4 points, and both possibilities do occur.

\section{Subgraph extension} \label{sec:hara}

This section concerns results on how we can extend knowledge about
identifiability of an induced subgraph to that of the original DAG.
We recall standard terminology in graphical modeling.  For a given DAG
$G = (V, E)$, we write $pa(v) = \{w : w \rightarrow v \in E\}$ for the
parent set of the node $v$, and
$ch(v)= \{ w : v \rightarrow w \in E\}$ for the child set of $v$.  If
for some node $s \in V$ there does not exist a node $s'\in V$ with
$s \rightarrow s' \in E$, then $s$ is a sink node.  If there is no
other node $s'\in V$ with $s' \rightarrow s \in E$, then $s$ is a
source node.  The following theorem is the main result of this
section.

\begin{theorem} \label{thm:subgraph_ext}
Given a DAG $G = (V, E)$, if there exists
\begin{enumerate}
\item a sink node $s \in V$ such that $pa(s) \not = V\setminus \{s\}$
  and the model $\mathcal{N}_*(G')$ of the induced subgraph $G'$ on
  $V\setminus \{s\}$ is generically finitely identifiable, or

\item  a source node $s \in V$ such that $ch(s) \not = V\setminus  \{s\}$ and the model $\mathcal{N}_*(G')$ of the induced subgraph $G'$ on $V\setminus \{s\}$ is generically finitely identifiable, 

\end{enumerate}
then the model $\mathcal{N}_*(G)$ is generically finitely identifiable.
\end{theorem}

Recall that in \tabref{experiment} there are $3344-2957 = 387$ DAGs
with $m=6$ nodes that are generically finitely identifiable but do not
satisfy our sufficient condition from \thmref{our_suff}.  The above
\thmref{subgraph_ext} provides a way to certify identifiability of
models falling within this ``gap'', provided that we have knowledge of
which DAGs on $m =5$ nodes are generically finitely identifiable.  For
instance, from our algebraic computations we know that there are $95 -
88 = 7$ DAGs that are generically finitely identifiable but cannot be
proven to be so by \thmref{our_suff}.  Of the $387$ aforementioned
DAGs on $6$ nodes, $194$ can be proven to be generically finitely
identifiable by using the knowledge about the $7$ graphs on $m=5$
nodes and applying \thmref{subgraph_ext}.  We remark that if a DAG
satisfies the condition in \thmref{our_suff}, the resulting supergraph
obtained by augmenting a sink (source) node that does not have every
other node as its parent (child) must also satisfy the condition in
\thmref{our_suff}.  Hence, given current state-of-the-art,
\thmref{subgraph_ext} is useful primarily as a tool to reduce the
identifiability problem to smaller subgraphs that may then be tackled
by algebraic methods.

%With \thmref{subgraph_ext} we may ask: Given we have knowledge of the generic finite identifiabilty statuses of all DAGs with $m = 5$ nodes, how many of the $387$ graphs above can we single out as generically finitely identifiable by applying \thmref{subgraph_ext}? It turns out that $194$ of them can have generic finite identifiability be told in this way.
% 

\thmref{subgraph_ext} is obtained by studying the maps $\phi_G$ and
$\varphi_G$ in \eqref{phi} and \eqref{varphi}. First consider
\eqref{phi}. In light of \lemref{localequalfinite}(ii), we can show
that $\phi_G$ is generically finite-to-one if there exists a proper
algebraic subset $\Xi \subset \mathbb{R}^{2m+|E|}$ such that
$|\mathcal{F}_{\phi_G|_{\Theta \setminus \Xi}} (\theta_0)| < \infty$
for all
$\theta_0 = (\Lambda_0, \Omega_0, \delta_0) \in \Theta \setminus \Xi$,
or equivalently,
\begin{equation}\label{coveqt_tetrad}
(I_m - \Lambda^T)\phi_G( \theta_0) (I_m - \Lambda)=  \Omega + \delta \delta^T,
\end{equation}
has finitely many solutions for $(\Lambda, \Omega, \delta)$ in
$\Theta \setminus \Xi$.  \emph{Throughout} this section, $\Xi$ is
taken so that all points
$(\Lambda, \Omega, \delta)\in\Theta\setminus \Xi$ have
$\delta_i\not=0$ for all $i=1,\dots,m$.  As such, the matrix
$\Omega + \delta \delta^T$ on the right hand side of
\eqref{coveqt_tetrad} has all entries nonzero and is known as a
Spearman matrix.

\begin{definition}\label{def:Spearman}
  A symmetric matrix $\Upsilon\in\mathbb{R}^{m\times m}$ of size $m \geq 3$ is a
  \emph{Spearman matrix} if $\Upsilon = \Omega + \delta \delta^T$ for
  a diagonal matrix $\Omega$ with positive diagonal and a vector
  $\delta$ with no zero elements.   
\end{definition}

Any Spearman matrix $\Upsilon$ is positive definite, and it is not
difficult to show that if $\Upsilon = \Omega + \delta \delta^T$ is
Spearman with $m \geq 3$ then the two summands $\Omega$ and
$\delta \delta^T$ are uniquely determined as rational functions of
$\Upsilon$.  Moreover, $\delta\delta^T$ determines $\delta$ up to sign
change.  For these facts see, for instance, Theorem 5.5 in
\cite{FactBerkeley}.  We term $\Omega$ the \emph{diagonal component}
of $\Upsilon$, and $\delta \delta'$ the \emph{rank-1 component}.  The
following theorem gives an implicit characterization of Spearman
matrices of size $m\ge 4$.

\begin{theorem} \label{thm:tetrad_char_Spear} A positive definite
  symmetric matrix
  $\Upsilon= (\upsilon_{ij})\in\mathbb{R}^{m\times m}$ of size
  $m\ge 4$ is a Spearman matrix if and only if, after sign changes of
  rows and corresponding columns, all its elements are positive and
  such that
%\begin{align}
%& \upsilon_{ij} \upsilon_{kl} - \upsilon_{ik} \upsilon_{jl} \label{tetrad1}\\
%&= \upsilon_{il} \upsilon_{jk} - \upsilon_{ik} \upsilon_{jl} \label{tetrad2}\\ 
% &=  \upsilon_{ij} \upsilon_{kl} - \upsilon_{il} \upsilon_{jk} \label{tetrad3} = 0  
%\end{align}
\begin{equation} \label{tetrad123}
 \upsilon_{ij} \upsilon_{kl} - \upsilon_{ik} \upsilon_{jl} 
\;=\; \upsilon_{il} \upsilon_{jk} - \upsilon_{ik} \upsilon_{jl}  
 \;=\;  \upsilon_{ij} \upsilon_{kl} - \upsilon_{il} \upsilon_{jk} \;=\; 0  
\end{equation}
for $i < j <k <l$, and
\begin{equation}\upsilon_{ii} \upsilon_{jk} - \upsilon_{ik} \upsilon_{ji} > 0 \label{tetrad4}\end{equation}
for $ i \not= j \not = k$.
\end{theorem}

This is essentially the same as Theorem 1 in \cite{spearmat}, which the reader is referred to for a proof. Unlike \cite{spearmat}, we have a strict inequality in \eqref{tetrad4} since in \defref{Spearman} we require the diagonal component of a Spearman matrix to be strictly positive.

%Let $\text{SPR}(m)$ denote the set of all $m \times m$ Spearman matrices. 
The three polynomial expressions in \eqref{tetrad123} are the
$2\times 2$ off-diagonal minors of the matrix $\Upsilon$, which are
also known as tetrads in the literature.  We call the quadruple
$i < j < k <l$ the indices of the tetrad they define.  Note that
\[
\upsilon_{ij} \upsilon_{kl} - \upsilon_{il} \upsilon_{jk} =
\left(\upsilon_{ij} \upsilon_{kl} - \upsilon_{ik} \upsilon_{jl}\right)
-\left(\upsilon_{il} \upsilon_{jk} - \upsilon_{ik}
  \upsilon_{jl}\right)
\]
so that the three tetrads in \eqref{tetrad123} are algebraically
dependent. In general, a symmetric $m \times m$ matrix $\Upsilon$ has
$2 {m \choose 4}$ algebraically independent tetrads and we write
$\text{TETRADS}(\Upsilon)$ to denote a column vector comprising a 
choice of $2 {m \choose 4}$ algebraically independent tetrads.
 
For each triple $(\Lambda, \Omega, \delta) \in \Theta \setminus \Xi$
that solves \eqref{coveqt_tetrad}, it must be true that
\begin{equation} \label{tetrad_eq}
\text{TETRADS} \left((I_m - \Lambda^T)\phi_G( \theta_0) (I_m -
  \Lambda) \right) = 0. 
\end{equation}
Together with the uniqueness of the diagonal and rank-1 components for
a Spearman matrix, if we can show only finitely many $\Lambda$'s solve
the system \eqref{tetrad_eq}, then we have shown that the model
$\mathcal{N}_*(G)$ is generically finitely identifiable. Our proof for
\thmref{subgraph_ext}$(i)$\ follows this approach.

Alternatively, based on \lemref{equiv_maps}, we can also prove generic
finite identifiability by considering the map $\varphi_G$ from
\eqref{varphi}.  We then need to show that there exists a proper algebraic subset
$\Xi\subset\mathbb{R}^{2m+|E|}$ 
so that
$|\mathcal{F}_{\varphi_G|_{\Theta\setminus \Xi}}(\theta_0)| < \infty$
for all
$\theta_0 = (\Lambda_0, \Psi_0, \gamma_0) \in \Theta \setminus \Xi$,
or equivalently,
\begin{equation}\label{coneqt_tetrad}
(I_m - \Lambda)^{-1} \varphi_G\left(\theta_0 \right) (I_m - \Lambda^T)^{-1}   = \Psi - \gamma \gamma^T
\end{equation}
has finitely many solutions for $(\Lambda, \Psi, \gamma)$ in
$\Theta \setminus \Xi$.  Again we assume that $\Xi$ is defined to
avoid issues due to zeros, that is, every triple
$(\Lambda, \Psi, \gamma)\in \Theta \setminus \Xi$ has $\gamma_i\not=0$
for all $i=1,\dots,m$.  We introduce the term coSpearman matrix to
describe the matrix on the right hand side of \eqref{coneqt_tetrad}.

\begin{definition}\label{def:coSpearman}
  A symmetric matrix $\Upsilon\in\mathbb{R}^{m\times m}$ of size $m \geq 3$ is a
  \emph{coSpearman matrix} if  $\Upsilon = \Psi -  \gamma \gamma^T$ for
  a diagonal matrix $\Psi$ with positive diagonal and a vector
  $\gamma$ with no zero elements.
\end{definition} 

Again, the diagonal component $\Psi$ and the rank-1 component
$\gamma \gamma^T$ are uniquely determined by $\Upsilon$; compare
\citet[p.~243]{stang1997}.
The following theorem is analogous to \thmref{tetrad_char_Spear}.

\begin{theorem} 
  \label{thm:tetrad_char_coSpear} A positive definite symmetric matrix
  $\Upsilon= (\upsilon_{ij})\in\mathbb{R}^{m\times m}$ of size
  $m\ge 4$ is a coSpearman matrix if and only if, after sign changes of
  rows and corresponding columns, all its non-diagonal elements are negative and such that 
%\begin{align}
%& \upsilon_{ij} \upsilon_{kl} - \upsilon_{ik} \upsilon_{jl} \label{cotetrad1}\\
%&= \upsilon_{il} \upsilon_{jk} - \upsilon_{ik} \upsilon_{jl} \label{cotetrad2}\\ 
% &=  \upsilon_{ij} \upsilon_{kl} - \upsilon_{il} \upsilon_{jk} \label{cotetrad3} = 0  
%\end{align}
\begin{equation} \label{cotetrad123}
 \upsilon_{ij} \upsilon_{kl} - \upsilon_{ik} \upsilon_{jl} 
\;=\; \upsilon_{il} \upsilon_{jk} - \upsilon_{ik} \upsilon_{jl}  
 \;=\;  \upsilon_{ij} \upsilon_{kl} - \upsilon_{il} \upsilon_{jk} \;=\; 0  
\end{equation}
for $i < j <k <l$, and
\begin{equation} \upsilon_{ii} \upsilon_{jk} - \upsilon_{ik} \upsilon_{ji} < 0 \label{cotetrad4} \end{equation}
for $ i \not= j \not = k$.
\end{theorem}

% \begin{proof}
% The theorem can be proved by mimicking the proof for \thmref{tetrad_char_Spear} as in \cite{spearmat}. 
% \end{proof}

Using the tetrad characterizations \eqref{cotetrad123} and the
uniqueness of diagonal and rank-1 components, one can now demonstrate
that the restricted map $\varphi_G|_{\Theta\setminus \Xi}$ has finite
fibers by showing that the system of tetrad equations 
% for the left
% hand side of \eqref{coneqt_tetrad},
\begin{equation}\label{tetrad_eq_con}
\text{TETRADS}\bigl((I_m - \Lambda)^{-1} \varphi_G\left(\theta_0 \right) (I_m - \Lambda^T)^{-1}\bigr) = 0
\end{equation}
admits only finitely many solutions for $\Lambda$ when
$\theta_0 \in \Theta \setminus \Xi$.
%Analogously, by the tetrad characterizations in \eqref{cotetrad123}, together with the uniqueness of the diagonal and rank-1 components, we can show generic finite identifiability by showing that the system of tetrad equations for the left hand side of \eqref{coneqt_tetrad}, 
%\begin{equation}\label{tetrad_eq_con}
%\text{TETRADS}\bigl((I_m - \Lambda)^{-1} \varphi_G\left(\Lambda_0, \Psi_0, \gamma_0 \right) (I_m - \Lambda^T)^{-1}\bigr) = 0,
%\end{equation}
%admits only finitely many $\Lambda$ as solutions.

The finiteness of solutions in $\Lambda$ for the system
\eqref{tetrad_eq}, or \eqref{tetrad_eq_con}, is a sufficient condition
for the generic finite identifiability of $\mathcal{N}_*(G)$. It is,
however, not obvious that these two systems necessarily have finitely
many solutions when $\mathcal{N}_*(G)$ is generically finitely
identifiable.  The following lemma states that such a converse does
hold for the following two types of DAGs, whose generic finite
identifiability can be easily checked by \thmref{our_suff}. Recall that the notation $``\subsetneq"$ means ``being a proper subset of".

\begin{lemma}\label{lem:sink_and_source}
Let $G = (V, E)$ be a DAG with vertex set $V=\{1,\dots,m\}$.
\begin{enumerate}
\item If $E \subsetneq \{(k, m) : k\le m-1\}$, then there exists
  a proper algebraic subset $\Xi$ such that for all
  $\theta_0 = (\Lambda_0, \Omega_0, \delta_0) \in \Theta \setminus
  \Xi$, the system
\begin{equation*} 
\text{TETRADS} \left((I_m - \Lambda^T)\phi_G\left( \theta_0\right) (I_m - \Lambda) \right) = 0 
\end{equation*}
is linear in the variable $\Lambda\in\mathbb{R}_E$ and is solved
uniquely by $\Lambda=\Lambda_0$.
\item If $E \subsetneq \{(1,k) : k\ge 2\}$, then there exists
  a proper algebraic subset $\Xi$ such that for all
  $\theta_0 = (\Lambda_0, \Psi_0, \gamma_0) \in \Theta \setminus
  \Xi$, the system
\begin{equation*} 
\text{TETRADS} \left((I_m - \Lambda)^{-1}\varphi_G\left(
    \theta_0\right) (I_m - \Lambda^T)^{-1} \right) = 0 
\end{equation*}
is linear in the variable $\Lambda\in\mathbb{R}_E$ and is solved
uniquely by $\Lambda=\Lambda_0$.
\end{enumerate}
\end{lemma}

The proof of \lemref{sink_and_source} is deferred to Appendix \ref{sec:proofs}.
 
\begin{proof}[Proof of \thmref{subgraph_ext}]
  We will first prove $(i)$, which uses \lemref{sink_and_source}$(i)$. The
  proof of $(ii)$ will follow from similar reasoning using
  \lemref{sink_and_source}$(ii)$.

  Without loss of generality, assume that the sink node $s = m$, by
  giving the nodes a new topological order if necessary.  Define two
  DAGs as follows.  First, let $G_1 = (V_1, E_1)$ be the subgraph of
  $G$ induced by the set
  $V_1 = V\setminus \{m\} = [m-1]$, where we adopt
  the shorthand $[k] := \{1, \dots,k\}$, $k\in\mathbb{N}$.  Second,
  let $G_2 = (V, E\setminus E_1)$ be the graph on $V$ obtained from
  $G$ by removing all edges that do not have the sink node $m$ as
  their head. As before, let
  $\Theta := \mathbb{R}_{E} \times \diag^+_{m} \times \mathbb{R}^{m}$.
  We will construct a proper algebraic subset $\Xi$, such that for any
  $\theta \in\Theta\setminus \Xi$, the fiber
  $\mathcal{F}_{\phi_G|_{\Theta \setminus \Xi}}(\theta)$ is finite.
  Then \lemref{localequalfinite}$(ii)$ applies and yields the
  assertion of \thmref{subgraph_ext}$(i)$.

  Let
  $\Theta_1 := \mathbb{R}_{E_1} \times \diag^+_{m-1} \times
  \mathbb{R}^{m-1}$,
  the open set on which the parametrization $\phi_{G_1}$ of model
  $\mathcal{N}_*(G_1)$ is defined.  By assumption, there exists a
  proper algebraic subset $\Xi_1'\subset\mathbb{R}^{2(m-1)+|E_1|}$
  such that the restricted map
  $\phi_{G_1}|_{\Theta_1 \setminus \Xi_1'}$ has finite fibers, by \lemref{localequalfinite}$(ii)$.  Extend $\Xi_1'$ to a proper
  algebraic subset of $\mathbb{R}^{2m+|E|}$ by defining
  \[
  \Xi_1  :=  \Xi_1' \times \mathbb{R}^{E \setminus E_1}
  \times \mathbb{R}^2,
  \]
  where $\mathbb{R}^{E \setminus E_1}$ accommodates the additional
  free variables $\lambda_{vm}$ with $v \in pa(m)$, and $\mathbb{R}^2$
  accommodates the two variables $\Omega_{mm}=\omega_m$ and
  $\delta_m$.

  Next, recall that for a given point
  $\theta' = (\Lambda', \Omega', \delta') \in \Theta$, any
  $(\Lambda, \Omega, \delta) \in \mathcal{F}_{\phi_G} (\theta')$ must
  satisfy the tetrad equations
  \begin{equation} \label{tetrad_sys_theta}
    \text{TETRADS} \left((I_m - \Lambda^T)\phi_G( \theta') (I_m -
      \Lambda) \right) = 0.
  \end{equation}
  Let $\lambda_{E_1} := (\lambda_{vw})^T_{(v, w) \in E_1}$.  Then
  any tetrad in~(\ref{tetrad_sys_theta}) with indices $i < j < k <m$
  has the form
  \[
  \sum_{m' \in pa(m)} a_{m'}\left( \lambda_{E_1}, \phi_G (\theta')\right)\lambda_{m'm} - b\left( \lambda_{E_1}, \phi_G (\theta')\right),
  \]
  where the $a_{m'}$ as well as $b$ are polynomials with the entries
  of $\lambda_{E_1}$ and the entries of a symmetric $m\times m$ matrix
  being their variables.  Let $\lambda_{pa(m), m}$ be the vector with
  entries $\lambda_{vm}$ for $v \in pa(m)$.  Then the part of the
  system \eqref{tetrad_sys_theta} involving the variables
  $\lambda_{v, m}$, $v\in pa(m)$, has the form
  \begin{equation} \label{linearsys}
    C\left(\lambda_{E_1}, \phi_G (\theta')  \right)\lambda_{pa(m), m}
    = c\left(\lambda_{E_1}, \phi_G (\theta')  \right), 
  \end{equation}
  where $C$ is a matrix of size $2 {m-1 \choose 3} \times |pa(m)|$,
  and $c$ is a vector of length $2 {m-1 \choose 3}$.  Both $C$ and $c$
  are filled with polynomials in the entries of $\lambda_{E_1}$ and a
  symmetric $m\times m$ matrix.  Since
  $(\Lambda, \Omega, \delta) \in \mathcal{F}_{\phi_G}(\theta')$, we
  have $\phi_G(\theta')=\phi_G(\Lambda, \Omega, \delta)$ and, thus,
  \begin{equation}  \label{linearsys2}
    C\left(\lambda_{E_1},\phi_G(\Lambda, \Omega, \delta)  \right)\lambda_{pa(m), m} = c\left(\lambda_{E_1}, \phi_G(\Lambda, \Omega, \delta)  \right).
  \end{equation}
  As $\theta'$ was an arbitrary point in $\Theta$, \eqref{linearsys2}
  holds for all $(\Lambda, \Omega, \delta) \in \Theta$.  We claim that
  $C\left(\lambda_{E_1},\phi_G(\Lambda, \Omega, \delta) \right)$ is of
  full rank for generic choices of $(\Lambda, \Omega, \delta) $.  To
  see this note that if $\lambda_{E_1}$ is set to $0$, then
  \eqref{linearsys2} becomes the system of tetrad equations for the
  graph $G_2$.  Using \lemref{sink_and_source}(i) and the assumption
  that $pa(m) \subsetneq V\setminus \{s\}$, we see that
  $C\left(\lambda_{E_1},\phi_G(\Lambda, \Omega, \delta) \right)$
  achieves full rank for $\lambda_{E_1}=0$ and a generic choice of
  $(\lambda_{pa(m), m}, \Omega, \delta )$.  We deduce that the rank
  is full generically.

  Let $\Xi_2$ be a proper algebraic subset such that
  $C\left(\lambda_{E_1},\phi_G(\Lambda, \Omega, \delta) \right)$ is of
  full rank for any
  $(\Lambda, \Omega, \delta) \in\Theta \setminus \Xi_2$.  Let $\Xi_3$
  be the (algebraic) set comprising all triples
  $(\Lambda, \Omega, \delta)$ with at least one coordinate
  $\delta_i=0$, and define $\Xi := \Xi_1 \cup \Xi_2 \cup \Xi_3$.
  Clearly, $\Xi$ is a proper algebraic subset of
  $\mathbb{R}^{2m+|E|}$.  Take $(\Lambda_0, \Omega_0, \delta_0)$ to be
  a point in $\Theta \setminus \Xi$ and define
  $\Sigma_0 := \phi_G(\Lambda_0, \Omega_0, \delta_0)$.  It remains to
  show that the  equation system
  \begin{equation} 
    \label{inpfsys} 
    \Sigma_0 \;=\; \phi_G(\Lambda, \Omega,
    \delta) = (I_m - \Lambda^T)^{-1}(\Omega + \delta \delta^T)(I_m -
    \Lambda)^{-1}
  \end{equation}
  has only finitely many solutions in $(\Lambda, \Omega, \delta)$ over
  the set $\Theta \setminus \Xi$.  

  We begin by observing that because $s=m$ is a sink node, by taking a
  submatrices in~(\ref{inpfsys}), we obtain the equation system
  \begin{align*} 
    ({\Sigma_0})_{[m-1]} 
    &= [(I_m - \Lambda^T)^{-1}(\Omega +
      \delta \delta^T)(I_m - \Lambda)^{-1}]_{ [m-1]} \\ 
    &= (I_{m-1} - {\Lambda}_{[m-1]}^T)^{-1}(\Omega_{[m-1]} +
      \delta_{[m-1]}\delta_{[m-1]}^T)(I_{m-1} - {\Lambda}_{[m-1]})^{-1}
    \\
    &= \phi_{G_1}\left( \Lambda_{[m-1]},
      \Omega_{[m-1]}, \delta_{[m-1]}\right). 
  \end{align*}
  Here, for an index set $W \subset [m]$, we write $x_W$ to denote the
  subvector $x_W=(x_v:v\in W)$ of vector $x= (x_1, \dots, x_m)^T$, and
  we similarly write $A_W$ for the $W\times W$ principal submatrix of
  a matrix $A$.  Let $\mathcal{S}\subset\Theta_1$ be the projection of
  the set of all triples $(\Lambda,\Omega,\delta)\in\Theta \setminus
  \Xi$ that solve~(\ref{inpfsys}) onto their triple of
  submatrices/subvector $(\Lambda_{[m-1]}, \Omega_{[m-1]},
  \delta_{[m-1]})$.  By choice of $\Xi$, we have that
  $\mathcal{S}\subset\Theta_1 \setminus \Xi_1'$ and, since
  $\phi_{G_1}|_{\Theta_1 \setminus \Xi_1'}$ has finite fibers, we know
  that $\mathcal{S}$ is finite.  However, a triple $(\Lambda_{[m-1]},
  \Omega_{[m-1]}, \delta_{[m-1]})\in\mathcal{S}$ determines the matrix
  $C$ and the vector $c$ in~(\ref{linearsys2}) and, by choice of
  $\Xi$, we may deduce that $\lambda_{pa(m), m}$ is uniquely
  determined by $(\Lambda_{[m-1]}, \Omega_{[m-1]}, \delta_{[m-1]})$.
  It follows that the solutions to~(\ref{inpfsys}) that are in
  $\Theta\setminus\Xi$ have their $\Lambda$ part equal to one of
  $|\mathcal{S}|/2$ many choices; recall that if $(\Lambda_{[m-1]},
  \Omega_{[m-1]}, \delta_{[m-1]})$ is in $\mathcal{S}$ then so is
  $(\Lambda_{[m-1]}, \Omega_{[m-1]},- \delta_{[m-1]})$.
  The proof is now complete because $\Lambda$ determines the Spearman
  matrix
  \[ 
  (I_m - \Lambda^T) \Sigma_0 (I_m - \Lambda ) = \Omega + \delta
  \delta^T,
  \]
  for which the diagonal component $\Omega$ and the rank-1 component
  $\delta \delta^T$ are uniquely determined.  Given the fact that
  $\delta \delta^T$ determines $\delta$ only up to sign,
  (\ref{inpfsys}) has $|\mathcal{S}|<\infty$ solutions over
  $\Theta\setminus\Xi$, which concludes the proof of $(i)$.

  \medskip
  The proof of $(ii)$ is analogous, and we only give a sketch.
  Instead of considering $\phi_G$ we turn to $\varphi_G$, which also
  has domain $\Theta$.  Without loss of generality, we let the source
  node be $s = 1$.  We then define $G_1 = (V_1, E_1)$ to be the
  subgraph of $G$ that is induced by $V_1 = \{2, \dots, m\}$, and we
  let $G_2 = (V, E \setminus E_1)$.  We consider the parametrization
  $\varphi_{G_1}$ with domain
  $\Theta_1 = \mathbb{R}_{E_1} \times \diag^+_{m-1} \times
  \mathbb{R}^{m-1}$.
  By assumption, $\mathcal{N}_*(G_1)$ is generically finitely
  identifiable, so there exists a proper algebraic subset $\Xi_1'$
  such that $\varphi_{G_1}|_{\Theta\setminus \Xi_1'}$ has finite fibers,
  by \lemref{localequalfinite}$(ii)$.

  On the other hand, for any $(\Lambda, \Psi, \gamma) \in \Theta$, we have
  \[
  \text{TETRADS}\left( (I_m - \Lambda)^{-1} \varphi_G  (\Lambda,
    \Psi, \gamma)  (I_m - \Lambda^T)^{-1}\right) = 0.
  \]
  Let $\lambda_{E_1} := (\lambda_{vw})^T_{(v, w) \in E_1}$ and
  $\lambda_{1, ch(1)} := (\lambda_{1v})^T_{v \in ch(1)}$.  Then the
  tetrad equations with one index equal to $s=1$ yield the equation system
  \begin{equation*} 
    \label{res_ted_sys2}
    C\left(\lambda_{E_1},\varphi_G  (\Lambda, \Psi, \gamma) \right) \lambda_{1, ch(1)}= c\left(\lambda_{E_1}, \varphi_G \left( \Lambda, \Psi, \gamma \right) \right),
  \end{equation*}
  where part $(ii)$ of \lemref{sink_and_source} can be applied to show
  that
  $C\left(\lambda_{E_1},\varphi_G (\Lambda, \Psi, \gamma) \right)$ is
  of full rank outside some proper algebraic subset $\Xi_2$.  We may
  then define a set $\Xi$ as in the proof of part $(i)$ and use
  arguments similar to the ones above for a proof of part $(ii)$ of
  our theorem.
\end{proof}

\section{Discussion} \label{sec:discussion}

% Our sufficient  condition in \thmref{our_suff} can certify the identifiability of a larger class of models than the sufficient condition of \propref{wermuth2005} by \cite{wermuth2005}, but it doesn't suggest an algorithm to algebraically solve for the parameter values $(\Lambda, \Omega, \delta)$ for a given generic covariance matrix $\Sigma$ of the observables. In addition, it remains an open question as to what the cardinality $k$ of a generic fiber is when a given model $\mathcal{N}_*(G)$ is known to be generically finitely identifiable. As per our discussion in \secref{intro}, the most intuitive guess for $k$ is $2$, where the non-injectivity of $\phi_G$ is solely due to a sign change in $\delta$. However, our exhaustive algebraic study on DAGs with $m = 4, 5, 6$ using Gr\"obner basis computations suggests that exceptions to this  do exist. This matter can be further pursued.

In this paper we studied identifiability of directed Gaussian
graphical models with one latent variable that is a common cause of
all observed variables.  To our knowledge, the best criteria to decide
on identifiability of such models are those given by
\cite{wermuth2005} who consider a more general setup of Gaussian
graphical models with one latent variable.  Their results provide a
sufficient condition for the strictest notion of identifiability that
is meaningful is this context, namely, whether the parametrization map
is generically 2-to-one.  Recall that the coefficients associated with
the edges pointing from the latent variable to the observables can
only be recovered up to a common sign change.

In our work, we take a different approach and study the Jacobian
matrix of the parametrization, which leads to graphical criteria to
check whether the parametrization is finite-to-one.  Our sufficient
condition covers all graphs that can be shown to have a 2-to-one
parametrization by the conditions of \cite{wermuth2005}.  However, our
sufficient condition, which is stated as \thmref{our_suff}, covers far
more graphs as was shown in the computational experiments in
Section~\ref{sec:example}.  Our \thmref{our_nec} describes a
complementary necessary condition.

By studying tetrad equations, we also give a criterion that allows one
to deduce identifiability of certain graphs from identifiability of
subgraphs (\thmref{subgraph_ext}).  This result is stated for generic
finite identifiability but as is clear from the proof, the result
would also confirm that the parametrization of a graph is generically
2-to-one provided the involved subgraph has a generically 2-to-one
parametrization.

The extension result from \thmref{subgraph_ext} can be used in
conjunction with the results obtained by the algebraic computations in
Section~\ref{sec:example}.  These computations solve the
identifiability problem for graphs with up to 6 nodes.  In particular,
we confirm that the sufficient conditions of \cite{wermuth2005} are
not necessary for the parametrization map to be generically 2-to-one
and provide examples of graphs that yield a generically finite but not
2-to-one parametrization.

As mentioned above, we studied models with one latent source $0$ that
is connected to all nodes that represent observed variables.  However,
the graphical criteria in \thmsref{our_suff} and \thmssref{our_nec}
can be readily extended to models with some of these factor loading
edges missing.  Given the previously used notation, we describe such
models as follows.  Let $G=(V,E)$ be a DAG with vertex set of size
$m=|V|$; these vertices index the observed variables.  Let
$V'\subset V$ be the nodes representing observed variables that do not
directly depend on the latent variable.  Then only the edges $0\to v$
with $v\in V\setminus V'$ are added when forming the extended DAG
$\overline{G}$.  The parametrization of the Gaussian graphical model
determined by $G$ and $V'$ is the restriction of $\phi_G$
from~(\ref{phi}) to the domain
\[
\Theta(V') := \left\{ (\Lambda,\Omega,\delta)\in \Theta: \delta_v=0
  \text{ 
  for all } v\in V' \right\}.
\]
When the parametrization maps $\tilde{\phi}_G$, ${\varphi}_G$ and
$\tilde{\varphi}_G$ are restricted to the same domain, the assertion
of \lemref{equiv_maps} still holds.  The corresponding identifiability
results, which are in the spirit of Corollary $1$ in
\cite{Grzebyk2004}, are stated below.  A brief outline of their proofs is
given in Appendix \ref{sec:proofs}.

\begin{theorem} [Sufficient condition]
  \label{thm:our_suff_ex} 
  Let $G= (V,E)$ be a DAG, and let $V'\subset V$.  If every connected
  component of $(G^c)_{V \setminus V'}$, the subgraph of $G^c$ induced
  by $V\setminus V'$, contains an odd cycle, then the parametrization
  map $\phi_G$ is generically finite-to-one when restricted to the
  domain $\Theta(V')$.
\end{theorem}

The necessary condition given next makes references to the graphs
$G_{con}$ and $G_{|L, cov}$ that were defined in the introduction.

\begin{theorem} [Necessary condition]\label{thm:our_nec_ex}
  Let $G= (V,E)$ be a DAG, and let $V'\subset V$.  In order for the
  restriction of $\phi_G$ to the domain $\Theta(V')$ to be generically
  finite-to-one, it is necessary that the following two conditions
  both hold:
  \begin{enumerate}
  \item Let
    $\widetilde{G}^c_{con} = (V \setminus V', \widetilde{E}_{con})$ be
    the subgraph of $G^c_{con}$ induced by $V \setminus V'$. If
    $d_{con}$ is the number of connected components in the graph
    $\widetilde{G}^c_{con}$ that do not contain any odd cycle, then
    $|\widetilde{E}_{con}| - |E| \geq d_{con}$.
  \item Let
    $\widetilde{G}^c_{|L, cov}= (V \setminus V', \widetilde{E}_{|L,
      cov})$
    be the subgraph of $G^c_{con}$ induced by $V \setminus V'$. If
    $d_{cov}$ is the number of connected components in the graph
    $\widetilde{G}^c_{|L, cov}$ that do not contain any odd cycle,
    then $|\widetilde{E}_{|L, cov}| - |E| \geq d_{cov}$.
  \end{enumerate}
\end{theorem}

While \thmsref{our_suff_ex} and \thmssref{our_nec_ex} may be useful in
some contexts, models in which latent variables are parents to only
some of the observables deserve a more in-depth treatment in future
work.  In particular, it would be natural to seek ways to
combine the results of \cite{wermuth2005} and the present paper with
the work of \cite{foygel2012} and \cite{drton:weihs:2015}.

\appendix
\section{Proofs}\label{sec:proofs}
\begin{proof}[Proof of \lemref{localequalfinite}]

  We may assume $ d \geq n$, otherwise $J_f$ is never of full column
  rank.  The implication $(i) \Rightarrow (ii)$ is obvious.

  To show $(ii) \Rightarrow (iii)$, suppose for contradiction that
  $J_f$ is not generically of full rank.  Since $f$ is polynomial, we
  then know that $\text{Rank}(J_f)= r < n$ generically, that is,
  outside a proper algebraic subset $S'\subset\mathbb{R}^n$ the rank
  is constant $r$.  By the rank theorem \citep[p.~229]{babyrudin}, for
  every point $s\in S\setminus (S'\cup \tilde S)$, we can choose an
  open ball $\mathcal{B}(s)$ that contains $s$, is a subset of
  $S\setminus (S'\cup \tilde S)$ and for which the restricted map
  $f|_{\mathcal{B}_s}$ has fibers of dimension $n - r > 0$,
  contradicting $(ii)$.

  It remains to show $(iii) \Rightarrow (i)$.  We observe that since
  $f$ is a polynomial we can assume $S = \mathbb{R}^n$.  We then show
  that the set of points with an infinite fiber, denoted
  \[
  \mathbb{F}_f := \{s \in \mathbb{R}^n :\left|\mathcal{F}_f(s)\right|
  = \infty\},
  \] 
  is contained in a proper algebraic subset of $\mathbb{R}^n$.  We
  note that it suffices to assume $n = d$, for without loss of
  generality, we can permute the $d$ component functions of $f$ and
  assume that
  $\pi \circ f : \mathbb{R}^n \longrightarrow \mathbb{R}^n$ has a
  generically full rank Jacobian matrix, where $\pi$ is the projection
  onto the first $n$ coordinates. Then
  $\mathbb{F}_f \subset \mathbb{F}_{\pi \circ f}$. 

  Now, assume $d = n$, and let
  $C = \{s \in \mathbb{R}^n : \det J_f(s) = 0\}$ be the set of
  critical points of $f$, where $J_f$ is the Jacobian matrix of $f$.
  Note that by assumption $C$ is a proper algebraic subset of
  $\mathbb{R}^n$.
%\begin{claim*}
%For a given point $y \in \mathbb{R}^n$, suppose $f^{-1}(\{y\}) \subset Z^c$, then $f^{-1}(\{y\})$ is finite.
%\end{claim*}
 \begin{claim*}
   If $y \in \mathbb{R}^n$ is a point such that
   $|\mathcal{F}_f(y)| = \infty$, then
   $\mathcal{F}_f(y) \cap C \not = \emptyset$.
 \end{claim*}
 \begin{proof}[Proof of the Claim]
   If an algebraic set like $\mathcal{F}_f(y)$ is infinite, then it
   has dimension $k>0$.  By semialgebraic stratification \citep{ARA},
   one can see that there exists an open set $U\subset\mathbb{R}^k$
   and a differentiable map $g: U \longrightarrow \mathcal{F}_f(y)$
   such that the Jacobian of $g$ has full rank on $U$.  If
   $\mathcal{F}_f(y) \cap C = \emptyset$, then the chain rule yields
   that the composition $f \circ g : U \rightarrow \{y\}$ has Jacobian
   of positive rank.  This, however, is a contradiction because
   $f \circ g$ is a constant function.  Hence,
   $\mathcal{F}_f(y) \cap C \not = \emptyset$.
 \end{proof}
 The claim implies that
 $\mathbb{F}_f \subset f^{-1}(f(C)) \subset f^{-1}(\overline{f(C)})$,
 where $\overline{f(C)}$ is the Zariski closure of the semialgebraic
 set $f(C)$. Since $\overline{f(C)}$ is algebraic, so is
 $ f^{-1}(\overline{f(C)})$ given that $f$ is a polynomial. To finish
 the proof we only need to show that $f^{-1}(\overline{f(C)})$ has
 dimension less than $n$, which is equivalent to
 $f^{-1}(\overline{f(C)}) \not= \mathbb{R}^n$. By Sard's theorem
 \citep[p.~192]{ARA}, $f(C)$, and thus also $\overline{f(C)}$, has
 dimension less than $n$. If $f^{-1}(\overline{f(C)}) = \mathbb{R}^n$,
 then the inverse function theorem, which says that the restricted map
 $f|_{\mathbb{R}^n\setminus C}$ is a local diffeomorphism, is
 contradicted.
\end{proof}

\begin{proof}[Proof of \thmref{Markov_equiv}]
  Let $m=|V|$.  For $i = 1, 2$, let
  $\overline{G}_i = (\overline{V}, \overline{E}_i)$ be the extended
  DAG of $G_i$, i.e., $\overline{V} = \{0, 1, \dots, m\}$, and
  $\overline{E}_i = E_i \cup \{0 \rightarrow v : v \in \{1, \dots,
  m\}\}$.
  By the well-known characterization that two DAGs are Markov
  equivalent if and only if they have the same skeleton and
  v-structures \citep{PearlCausality}, it is easy to see that
  $\overline{G}_1$ and $\overline{G}_2$ are also Markov
  equivalent. 

  For $i\in\{1,2\}$, let $\Theta_i  := \mathbb{R}_{E_i} \times
  \diag_m^+ \times \mathbb{R}^m$.  Define
  \[
  \Phi_{\overline{G}_i} \bigl((\Lambda, \Omega, \delta) \bigr) =
  (I_{m+1} - \overline{\Lambda}^T)^{-1} \overline{\Omega} (I_{m+1} -
  \overline{\Lambda})^{-1} ,
  \]
  where $\Theta_i  := \mathbb{R}_{E_i} \times \diag_m^+ \times
  \mathbb{R}^m$,  $\overline{\Lambda}$ is a $(m+1) \times (m+1)$
  matrix such that 
  \[
  \overline{\Lambda}_{vw} =
  \begin{cases}
    \delta_w & \text{if } v = 0, w  = 1, \dots, m, \\
    \Lambda_{vw} & \text{if } v, w = 1, \dots, m, \\
    0 &  \text{otherwise}, \\
  \end{cases}
  \]
  and $\overline{\Omega}$ is a diagonal matrix with
  $\overline{\Omega}_{00} = 1$ and
  $\overline{\Omega}_{vv} = \Omega_{vv}$ for $v = 1, \dots m$.  Then
  the image $\Phi_{\overline{G}_i}(\Theta_i)$ is the set of all
  covariance matrices of $(m+1)$-variate Gaussian distributions that
  obey the global Markov property of $\overline{G}_i$ and have the
  variance of node $0$, which represents the latent variable $L$,
  equal to $1$.  Consider the projection 
  \[
  \pi(\Sigma) = \Sigma_{\{1, \dots, m\}, \{1, \dots, m\}},
  \]
  where $\Sigma$ has its rows and columns indexed by $\{0,\dots,m\}$.
  Then the parametrization map for the latent variable model
  $\mathcal{N}_*(G_i)$ equals
  \begin{equation}
    \label{eq:latent-parametrization}
    \phi_{G_i} = \pi \circ \Phi_{\overline{G}_i}.
  \end{equation}

  Since $\overline{G}_1$ and $\overline{G}_2$ are Markov equivalent,
  $\Phi_{\overline{G}_1}(\Theta_1)=\Phi_{\overline{G}_2}(\Theta_2)$.
  By \lemref{obs_ident}, each map $\Phi_{\overline{G}_i}$ is injective
  on $\Theta_i$ with rational inverse defined on the common image
  $\Phi_{\overline{G}_1}(\Theta_1)=\Phi_{\overline{G}_2}(\Theta_2)$.
  From~(\ref{eq:latent-parametrization}), we obtain that
  \[
  \phi_{G_1} = \pi \circ \Phi_{\overline{G}_1} = \pi \circ
  \Phi_{\overline{G}_2} \circ \Phi_{\overline{G}_2}^{-1} \circ
  \Phi_{\overline{G}_1} = \phi_{G_2} \circ \left(\Phi_{\overline{G}_2}^{-1} \circ
  \Phi_{\overline{G}_1} \right).
  \]
  Since
  $\Phi_{\overline{G}_2}^{-1} \circ
  \Phi_{\overline{G}_1}:\Theta_1\longrightarrow \Theta_2$
  is a diffeomorphism, the chain rule implies that the Jacobian of
  $\phi_{G_1}$ can be of full column rank if and only if the same is
  true for $\phi_{G_2}$.  Since $\phi_{G_i}$ are polynomial, the two
  Jacobians either both have generically full rank or are both
  everywhere rank deficient.  By \lemref{localequalfinite},
  $\phi_{G_1}$ is generically finite-to-one if and only if
  $\phi_{G_2}$ is so.
\end{proof}

\begin{proof} [Proof of \lemref{Jacob_structure}]
%   Order the elements of each one of the sets $D$,  $E$ and $N$
%   lexicographically, that is, for $D$, $(v, v) < (w, w)$ if $v < w$, and for $E$ or $N$, a pair $(v, w) < (v', w')$ if either \begin{inparaenum} \item $v < v'$, or if \item $v=v'$ and $w < w'$ \end{inparaenum}. Without loss of generality, we assume that the elements of each block row/column index set are arranged as follows: 
% \begin{enumerate}[]\item $\Psi$: $\psi_v$ is arranged to the left of $\psi_w$ if $v < w
% $\item $\Lambda$:  $\lambda_{vw}$ is arranged to the left of $\lambda_{ ux}$ if $(v, w) < (u, x)$
% \item $\gamma$: $\gamma_v$ is arranged to the left of $\gamma_w$ when $v < w$
% \item D: $(v, v)$ is listed before $(w, w)$ if $v < w$
% \item E: $v \rightarrow w$ is listed before $u \rightarrow x$ if $(v, w) < (u, x)$ 
% \item N: $v \not \rightarrow w$ is listed before $u \not \rightarrow x$ if $(v, w) < (u, x)$ 
% \end{enumerate}
We  first give the structure of $J(\tilde{\varphi}_G)$ block by block.
\begin{enumerate}[(a)]
\item ``$[J(\tilde{\varphi}_G)]_{D, \{\Psi, \Lambda, \gamma\}}$": 
For a given pair $(v, v) \in D$, \[[\tilde{\varphi}_G (\Lambda, \Psi, \gamma)]_{vv}  =\psi_v + \left(\displaystyle\sum_{w : v \rightarrow w \in E} \psi_w \lambda_{vw}^2 \right)- \gamma_v^2. \]
Hence, 
\begin{equation} \label{middle1}
 [J(\tilde{\varphi}_G)]_{(v, v), \psi_w} =
  \begin{cases}
1 & \text{if } v = w,\\
  \lambda_{vw}^2  & \text{if } v \rightarrow w \in E, \\
   0 & \text{otherwise}, \\
    \end{cases}
\end{equation} 
\begin{equation} 
 [J(\tilde{\varphi}_G)]_{(v, v), \lambda_{wu}} =
  \begin{cases}
2 \lambda_{wu} \psi_u \ & \text{if } v = w,\\
  0 & \text{otherwise}, \\
    \end{cases}
\end{equation}  
and
\begin{equation} 
 [J(\tilde{\varphi}_G)]_{(v, v), \gamma_u} =
  \begin{cases}
 -2 \gamma_u& \text{if } v = u,\\
  0 & \text{otherwise}. \\
    \end{cases}
\end{equation}

\item ``$[J(\tilde{\varphi}_G)]_{E, \{\Psi, \Lambda, \gamma\}}$":
For any $v \rightarrow w\in E$, 
\[
[\tilde{\varphi}_G (\Lambda, \Psi, \gamma)]_{vw} = -\lambda_{vw} \psi_w +  \displaystyle \left(\sum_{ u : \substack{ v \rightarrow u \in E\\ w\rightarrow u \in E} }\lambda_{vu} \lambda_{wu} \psi_u\right) - \gamma_v \gamma_w.
\] 
Hence, 
\begin{equation} 
 [J(\tilde{\varphi}_G)]_{v \rightarrow w, \psi_u} =
  \begin{cases}
-\lambda_{vw} & \text{if } u = w,\\
  \lambda_{vu} \lambda_{wu} & \text{if } v \rightarrow u \in E \text{ and } w \rightarrow u \in E, \\
   0 & \text{otherwise}, \\
    \end{cases}
\end{equation} 
\begin{equation} \label{middle2}
 [J(\tilde{\varphi}_G)]_{v \rightarrow w, \lambda_{ux}} =
  \begin{cases}
-\psi_w& \text{if } v = u, w = x,\\
  \lambda_{wx} \psi_x & \text{if } u = v \text{, }  u \rightarrow x \in E \text{ and } w \rightarrow x \in E,\\
  \lambda_{vx} \psi_x & \text{if } u = w\text{, } u \rightarrow x \in E \text{ and } v \rightarrow x \in E, \\
   0 & \text{otherwise}, \\
    \end{cases}
\end{equation} 
and 
\begin{equation}
 [J(\tilde{\varphi}_G)]_{v \rightarrow w, \gamma_u} =
  \begin{cases}
-\gamma_w& \text{if } v = u,\\
-\gamma_v & \text{if } w = u, \\
  0 & \text{otherwise}. \\
    \end{cases}
\end{equation} 
\item ``$[J(\tilde{\varphi}_G)]_{N, \{\Psi, \Lambda, \gamma\}}$": For any $v \not \rightarrow w \in N$, \begin{equation}[\tilde{\varphi}_G (\Lambda, \Psi, \gamma)]_{vw} = \displaystyle \left(\sum_{u :  \substack{ v \rightarrow u \in E\\ w\rightarrow u \in E} } \lambda_{vu} \lambda_{wu}\psi_u  \right)- \gamma_v \gamma_w. \end{equation} Hence,
\begin{equation}
 [J(\tilde{\varphi}_G)]_{v \not\rightarrow w, \psi_u} =
  \begin{cases}
\lambda_{vu} \lambda_{wu}& \text{if } v \rightarrow u \in E\text{ and } w \rightarrow u \in E, \\
   0 & \text{otherwise}, \\
    \end{cases}
\end{equation}
\begin{equation}
 [J(\tilde{\varphi}_G)]_{v \not\rightarrow w, \lambda_{ux}} =
  \begin{cases}
  \lambda_{wx} \psi_x & \text{if } u = v \text{, }  u \rightarrow x \in E \text{ and } w \rightarrow x \in E,\\
  \lambda_{vx} \psi_x & \text{if } u = w\text{, } u \rightarrow x \in E \text{ and } v \rightarrow x \in E ,\\
   0 & \text{otherwise}, \\
    \end{cases}
\end{equation}
and
\begin{equation} \label{gamma_stuff}
 [J(\tilde{\varphi}_G)]_{v \not \rightarrow w, \gamma_u} =
  \begin{cases}
-\gamma_w& \text{if } v = u,\\
-\gamma_v & \text{if } w = u, \\
  0 & \text{otherwise}. \\
    \end{cases}
\end{equation}
\end{enumerate}

With slight abuse of notation, let $|\Psi|$, $|\gamma|$, $|\Lambda|$
denote the number of free variables in $\Psi$, $\gamma$ and $\Lambda$
respectively. Considering that $|D| = |\Psi|$ and $|E| = |\Lambda|$,
we must have that $|N| \geq |\gamma|$ since $J(\tilde{\varphi}_G)$ is
a tall matrix.  Hence, if $[J(\tilde{\varphi}_G)]_{N, \gamma}$ is
generically of full column rank, then there exists a subset
$N' \subset N$ such that $|N'| =|\gamma|$ and the determinant of
$J(\tilde{\varphi}_G)_{N', \gamma}$ is a nonzero polynomial in the
variables of $\gamma$, in consideration of \eqref{gamma_stuff}.  Now
it suffices to show that the $(2m + |E|) \times (2m + |E|)$ square
submatrix
$[J(\tilde{\varphi}_G)]_{\{D, E, N'\}, \{\Psi, \Lambda, \gamma\}}$ is
generically of full rank.  

Since the concerned matrix has polynomial entries, we need to show
that the determinant of
$[J(\tilde{\varphi}_G)]_{\{D, E, N'\}, \{\Psi, \Lambda, \gamma\}}$ is
a nonzero polynomial.  To this end, it is sufficient to show that the
determinant is a nonzero polynomial in the entries of
$(\Lambda, \gamma)$ when we specialize $\psi_1 = \dots = \psi_m = 1$.
Noting that $|\Psi|+ |\Lambda| + |\gamma|= |D| + |E| + |N'|$, let $P$
denote the set of all permutation functions mapping from the set
$D \cup E \cup N'$ to the set of free variables in $\Lambda$, $\Psi$
and $\gamma$.  Choose any ordering of the elements of domain and
codomain so as to have a well-defined sign for the permutations.  Then
by Leibniz's formula, we have
\[
\det\left([J(\tilde{\varphi}_G)]_{\{D, E, N'\}, \{\Psi, \Lambda,
    \gamma\}}\right) \notag = \displaystyle \sum_{\sigma \in P}
\text{sgn}(\sigma) \prod_{s \in D \cup E \cup N'}
J(\tilde{\varphi}_G)_{s, \sigma(s)} .
\]
Let $\tilde{P}$ be the subset of all permutations $\sigma\in P$ with
$\sigma((v, v)) = \psi_v$ for all $(v, v) \in D$ and
$\sigma((v, w)) = \lambda_{vw}$ for all $(v, w) \in E$.  Then we
obtain that
\begin{align}
&\det\left([J(\tilde{\varphi}_G)]_{\{D, E, N'\}, \{\Psi, \Lambda, \gamma\}}\right) \notag\\
&=  \sum_{\sigma \in \tilde{P}} \text{sgn}(\sigma)\prod_{s \in D \cup E \cup N'} J(\tilde{\varphi}_G)_{s, \sigma(s)}  + \sum_{\sigma \in P\setminus \tilde{P}} \text{sgn}(\sigma) \prod_{s \in D \cup E \cup N'} J(\tilde{\varphi}_G)_{s, \sigma(s)}\notag\\
& = \pm \det\bigl(J(\tilde{\varphi}_G)_{N', \gamma}\bigr)  + \sum_{\sigma \in P\setminus \tilde{P}} \text{sgn}(\sigma) \prod_{s \in D \cup E \cup N'} J(\tilde{\varphi}_G)_{s, \sigma(s)}, \label{lasteq}
\end{align}
where the equality in \eqref{lasteq} follows from \eqref{middle1},
\eqref{middle2} and the fact that $\psi_1 = \dots = \psi_m = 1$.
We also deduce from
\eqref{middle1}-\eqref{gamma_stuff} that every summand in the second
term of \eqref{lasteq} is either zero or a polynomial term involving
free variables of $\Lambda$.  In contrast,
$\det\bigl(J(\tilde{\varphi}_G)_{N', \gamma}\bigr)$ is a nonzero
polynomial only in free variables of $\gamma$ and can thus not be
canceled by the second term in \eqref{lasteq}.
\end{proof}

\begin{proof}[Proof for \lemref{sink_and_source}]
  We first prove $(i)$.  
  Since $\mathcal{N}_*(G)$ is generically finitely identifiable by
  \thmref{our_suff}, there exists an algebraic subset $\Xi'$ such that
  for all $\theta \in \Theta \setminus \Xi'$,
  $|\mathcal{F}_{\phi_G}(\theta)| < \infty$.  Define $\Xi$ to be the
  union of $\Xi'$ and the set of triples
  $(\Lambda,\Omega,\delta)\in\mathbb{R}^{2m+|E|}$ with at least one
  coordinate $\delta_i=0$.  Let
  $\Sigma_0 = \phi_G(\Lambda_0, \Omega_0, \delta_0)$ and
\begin{equation} \label{tetradentires}
S  = (s_{ij}) := (I_m - \Lambda^T)\Sigma_0 (I_m - \Lambda).
\end{equation}
Then for $1 \leq i< j\leq m$,
\begin{align*}
s_{ij} &= \sum_{1 \leq k, k' \leq m} \lambda_{ki} [\Sigma_0]_{k k'}\lambda_{k'j} - \sum_{1 \leq k \leq m} [\Sigma_0]_{i k} 
\lambda_{kj} - \sum_{1 \leq k \leq m} \lambda_{ki}[\Sigma_0]_{k j} + [\Sigma_0]_{i j} \\ 
&= 
  \begin{cases}
     - \displaystyle\sum_{(k, m) \in E} [\Sigma_0]_{i k} 
\lambda_{km}  + [\Sigma_0]_{i m}& \text{if } j = m, \\
   [\Sigma_0]_{i j}       & \text{if }  j < m,
  \end{cases},
\end{align*}
where the last equality follows from the fact that $\lambda_{ij}$ are
nonzero only when $(i, j) \in E$.  Hence, for any four indices
$1 \leq i < j < k < l \leq m$, the tetrads
\begin{align*}
s_{ij}s_{kl} - s_{ik}  s_{jl}, && 
s_{il}s_{jk} - s_{ik}  s_{jl}
\end{align*}
are constant polynomials when $l < m$ and have degree $1$ in the
variables $\{\lambda_{vm} : (v, m)\in E\}$ when $l = m$.  The equation
system
\begin{equation*} 
  \text{TETRADS} (S) = 0 
\end{equation*} 
is a thus a consistent linear system that can be represented as
\begin{equation} \label{linsys}
  C \lambda_{pa(m), m} = c,
\end{equation}
where $\lambda_{pa(m), m} = (\lambda_{vm})_{v \in pa(m)}^T$ is the
vector of all free $\Lambda$ variables, $C$ is a
$2 {m - 1\choose 3} \times |pa(m)|$ matrix and $c$ is a
$2 {m - 1\choose 3}$-vector.  Both $C$ and $c$ depend only
on $\Sigma_0$.

To finish the proof, we now need to show that \eqref{linsys} is
uniquely solvable in $\lambda_{pa(m), m}$. We will aim to contradict
$|\mathcal{F}_{\phi_G}(\theta_0)| < \infty$ if \eqref{linsys} does not
have a unique solution.  Note that the solution set is an affine
subspace $\mathcal{L}\subset\mathbb{R}^{|E|}$.  For a contradiction,
suppose that $\mathcal{L}$ is of positive dimension.
% Note
% $\lambda^0_{pa(m), m} = ( \lambda^0_{vm})^T_{v \in pa(m)}\in
% \mathcal{S}$
% since it must solve the system. 
Upon substituting $\Lambda = \Lambda_0$
into \eqref{tetradentires}, we obtain
\[S_0  = (s^0_{ij}) \;=\;  (I_m - \Lambda_0^T)\Sigma_0 (I_m - \Lambda_0), \]
and in consideration of \eqref{tetrad4} in \thmref{tetrad_char_Spear}, it must be true that 
\begin{equation*}\label{fixedineq}
  s^0_{ii} s^0_{jk}  - s^0_{ik} s^0_{ji} > 0\text{, \ for all } i \not= j \not= k.\end{equation*}
We may then pick an open ball $\mathcal{B}(\Lambda_0)$ such that for
all solutions $\Lambda\in\mathcal{L}\cap\mathcal{B}(\Lambda_0)$, the
matrix $S=(s_{ij})$ defined by 
\eqref{tetradentires} 
satisfies
\[s_{ii} s_{jk} - s_{ik} s_{ji} > 0\text{ , for all } i \not= j \not=
k.
\]
It follows that $\mathcal{L}\cap\mathcal{B}(\Lambda_0)$ is an infinite set
whose elements $\Lambda$ all make the matrix
$(I_m -\Lambda^T) \Sigma_0 (I_m -\Lambda) $ a Spearman matrix.  Hence,
the system
\[ 
(I_m - \Lambda^T) \Sigma_0 (I_m - \Lambda) = \Omega + \delta \delta'
\]
has infinitely many solutions, contradicting
$|\mathcal{F}_{\phi_G}(\theta_0)| < \infty$.

The proof of $(ii)$ is analogous.  We first let $\Upsilon_0 = \varphi_G (\Lambda_0, \Psi_0, \gamma_0)$ and define
\begin{equation}
\tilde{S}  = (\tilde{s}_{ij}) \;=\; (I_m - \Lambda)^{-1}
\Upsilon_0(I_m - \Lambda^T)^{-1}. 
\end{equation}
Noting that in this case $(I_m - \Lambda)^{-1} = I_m + \Lambda$, it can be easily seen that 
\[\text{TETRADS}(\tilde{S}) = \text{TETRADS} \bigl(  (I_m - \Lambda)^{-1} \Upsilon_0(I_m - \Lambda^T)^{-1} \bigr) = 0\]
is a linear system in the variables $\{ \lambda_{1v}: v \in ch(1)\}$.
Similar to the above arguments, we may use \thmref{our_suff} and
\thmref{tetrad_char_coSpear} to prove by contradiction that the system
can only have a unique solution in $\{ \lambda_{1v}: v \in ch(1)\}$.
\end{proof}

\begin{proof}[Proof of \thmsref{our_suff_ex} and \thmssref{our_nec_ex}]
  For \thmref{our_suff_ex}, one can partition the Jacobian matrix
  $J(\tilde{\varphi}_G)$ of $\tilde{\varphi}_G$ as in \eqref{Jacob},
  only with $\gamma$ replaced by
  $\gamma_{V\setminus V'} = \{\gamma_v : v \in V\setminus V'\}$.  In
  analogy with \lemref{Jacob_structure}, it can be shown that
  $J_{\tilde{\varphi}_G}$ is of column full rank if
  $[J(\tilde{\varphi}_G)]_{N, \gamma_{V\setminus V'}}$ is.  The
  reasoning is then analogous to that in the proof of
  \thmref{our_suff}, the main step being the application of 
  \lemref{lem_vicard} where the graph defining the considered map becomes
  $(G^c)_{V \setminus V'}$.

  The proof of \thmref{our_nec_ex} is analogous to the proof of
  \thmref{our_nec}.  The only change is to replace $G_{con}^c$,
  $G_{|L, cov}^c$, $\gamma$ and $\delta$ by $\widetilde{G}_{con}^c$,
  $\widetilde{G}_{|L, cov}^c$, $\gamma_{V\setminus V'} $ and
  $\delta_{V\setminus V'}$, respectively.   
  % Details are left to the
  % readers.
\end{proof}

\section*{Acknowledgments}
We thank Robin Graham and S\'andor Kov\'acs for helpful comments on
the proof of \lemref{localequalfinite}. This work was partially
supported by the U.S. National Science Foundation (DMS-1305154), the
U.S. National Security Agency (H98230-14-1-0119), and the University
of Washington's Royalty Research Fund.  The United States Government
is authorized to reproduce and distribute reprints.

\bibliographystyle{ba}
\bibliography{iden_bib}

\end{document}